\RequirePackage{fix-cm}
\documentclass{amsart} 
\usepackage[utf8]{inputenc}
\usepackage{mathptmx}
\usepackage{latexsym}
\usepackage{amsmath}
\usepackage{amssymb}
\usepackage{amsfonts} 
\usepackage{eucal} 
\usepackage{amsbsy}
\usepackage{amsthm}
\usepackage[all]{xy}
\usepackage[pagebackref=true]{hyperref}
\usepackage{graphicx}
\usepackage{tikz}
\usepackage{color, colortbl}
\usepackage{stmaryrd}

\definecolor{Gray}{gray}{0.9}

\definecolor{Gray1}{gray}{0.9}
\definecolor{Gray2}{gray}{0.8}
\definecolor{Gray3}{gray}{0.7}

\newtheorem{thm}{Theorem}[section]
\newtheorem*{thm*}{Theorem}
\newtheorem{lemma}[thm]{Lemma}
\newtheorem{prop}[thm]{Proposition}
\newtheorem{cor}[thm]{Corollary}
\newtheorem*{cor*}{Corollary}

\theoremstyle{definition}
\newtheorem{defn}[thm]{Definition}
\newtheorem{example}[thm]{Example}
\newtheorem{examples}[thm]{Examples}

\theoremstyle{remark}
\newtheorem{remark}[thm]{Remark}

\newcommand {\Ga}    {\ensuremath{\mbox{$\mathcal{G}$}}}

\newcommand {\Sa}    {\ensuremath{\mbox{$\mathcal{S}$}}}

\newcommand {\Xa}    {\ensuremath{\mathcal{X}}}
\newcommand {\Ya}    {\ensuremath{\mathcal{Y}}}
\newcommand {\Za}    {\ensuremath{\mathcal{Z}}}

\newcommand {\real}  {\ensuremath{\mathbb{R}}}

\newcommand {\intg}  {\ensuremath{\mathbb{Z}}}

\newcommand {\cplx}  {\ensuremath{\mathbb{C}}}
\newcommand {\rat}   {\ensuremath{\mathbb{Q}}}

\newcommand {\Hom}   {\ensuremath{\operatorname{Hom}}}
\newcommand {\End}   {\ensuremath{\operatorname{End}}}

\newcommand {\codim} {\ensuremath{\operatorname{codim}}}

\newcommand {\smlhf} {\ensuremath{\mbox{$\frac{1}{2}$}}}

\newcommand {\supp}  {\ensuremath{\operatorname{supp}}}

\newcommand {\pt}    {\ensuremath{\operatorname{pt}}}
\newcommand {\id}    {\ensuremath{\operatorname{id}}}

\newcommand {\Witt}   {{\ensuremath{\operatorname{Witt}}}}

\newcommand {\MWITT}   {\ensuremath{\operatorname{MWITT}}}

\newcommand {\KO}   {{\ensuremath{\operatorname{KO}}}}

\newcommand {\BM}    {{\ensuremath{\operatorname{BM}}}}

\newcommand {\bp}    {{\ensuremath{\bar{p}}}}

\newcommand {\bm}    {{\ensuremath{\bar{m}}}}

\newcommand {\wY}    {{\ensuremath{\widetilde{Y}}}}

\newcommand {\wB}    {{\ensuremath{\widetilde{B}}}}
\newcommand {\Sign}    {{\ensuremath{\operatorname{Sign}}}}
\newcommand {\sig}    {{\ensuremath{\operatorname{sign}}}}

\newcommand {\bsd}    {{\ensuremath{\operatorname{bsd}}}}
\newcommand {\tr}    {{\ensuremath{\operatorname{tr}}}}
\newcommand {\Bdy}    {{\ensuremath{\operatorname{Bdy}}}}
\newcommand {\Lk}    {{\ensuremath{\operatorname{Lk}}}}

\newcommand {\sA}    {\mathbf{A}}
\newcommand {\sS}    {\mathbf{S}}
\newcommand {\sIS}   {\mathbf{IS}}

\makeatletter
\@namedef{subjclassname@2020}{%
  \textup{2020} Mathematics Subject Classification}
\makeatother

\begin{document}


\title[Equivariant L-Classes]
  {Equivariant L-Classes of Atiyah-Singer-Zagier Type for Singular Spaces}

\author{Markus Banagl}

\address{Institut f\"ur Mathematik, Universit\"at Heidelberg,
  Im Neuenheimer Feld 205, 69120 Heidelberg, Germany}

\email{banagl@mathi.uni-heidelberg.de}

\thanks{This work is funded in part by a research grant of the
 Deutsche Forschungsgemeinschaft (DFG, German Research Foundation)
 -- Projektnummer 495696766, and was supported in addition by
 INdAM, Istituto Nazionale di Alta Matematica}

\date{April 2026}

\subjclass[2020]{55N33, 57N80, 57R20, 57R91, 57S17}


\keywords{Stratified Spaces, Intersection Homology,
 Characteristic Classes, Transformation groups, $G$-Signature}


\begin{abstract}
If a finite group $G$ acts on a rational homology manifold, then the orbit space is well-known to be a rational homology manifold again. We consider here actions on spaces that may be much more singular. If the $G$-space is a Witt pseudomanifold, which includes all arbitrarily singular complex pure-dimensional algebraic varieties, then we prove that the orbit space is again a Witt pseudomanifold. In the compact oriented situation, this implies that the orbit space possesses characteristic L-classes, as defined by Goresky and MacPherson. 
We then construct Atiyah-Singer-Zagier type equivariant L-classes 
for such $G$-pseudomanifolds
which serve, as we show by establishing an averaging formula,
as a tool to compute the Goresky-MacPherson
L-class of the orbit space. The construction of the equivariant class builds on intersection homological transfer properties and on recent joint K-theoretic work with Eric Leichtnam and Paolo Piazza, which established a G-signature theorem on Witt pseudomanifolds.
\end{abstract}

\maketitle


\tableofcontents


\section{Introduction}

Let $X$ be a compact oriented pseudomanifold (without boundary).
We place the Witt condition on $X$, which requires that the middle perversity
intersection chain complex of sheaves is self-dual. For example every
pure-dimensional compact complex algebraic variety is a Witt pseudomanifold.
Suppose that a finite group $G$ acts on $X$ by orientation preserving
transformations. Then for every element $g\in G$, we define an
\emph{equivariant $L$-class}
\[ L_* (g,X) \in H_* (X;\cplx)^G.  \]
For $g=1$, this class agrees with the Goresky-MacPherson-Siegel
$L$-class of $X$ (\cite{gmih1}, \cite{siegel}), $L_* (1,X) = L_* (X)$.
When $X$ is a rational homology manifold, $L_* (g,X)$ agrees with the
Poincar\'e dual of Zagier's cohomological class $L^* (g,X) \in H^* (X;\cplx)^G$
constructed in \cite{zagier} and \cite{zagiertopology}.
Suppose that $X$ is a smooth manifold such that
$G$ acts smoothly on $X$, preserving the orientation.
For an element $g\in G$, let $X^g$ be the submanifold of points fixed by $g$
and let $\operatorname{or}_{X^g}$ denote the orientation local system on $X^g$.
Atiyah and Singer constructed an equivariant class
in $H^* (X^g; \operatorname{or}_{X^g} \otimes \cplx)$,
explicitly given in terms of characteristic
classes of $X^g$ and of the equivariant normal bundle of the inclusion
$j: X^g \subset X$, see \cite[p. 582]{atiyahsinger}. 
Let $L' (g,X)$ be the formulation of these classes
as given by Zagier in \cite[p. 12, (27)]{zagier}, \cite{zagiertopology}.
The $G$-Signature Theorem of Atiyah and Singer is the relation
\[ \Sign (g,X) = \langle L' (g,X), [X^g] \rangle. \] 
The inclusion has an associated Gysin homomorphism
$j_!: H^* (X^g) \to H^* (X)$ and the image of the Atiyah-Singer-class under $j_!$
is Zagier's class, i.e. $j_! L'(g,X) = L^* (g,X),$
\cite[p. 4, (3); p. 21, Thm. 1]{zagier}.
As the Gysin map corresponds under Poincar\'e duality to covariant
homological pushforward $j_*: H_* (X^g) \to H_* (X),$ the Atiyah-Singer-class
is related to our homological class by
\[ j_* (L'(g,X) \cap [X^g]) = L_* (g,X)  \]
in the differentiable case.

As in the smooth situation, an important application of the equivariant
classes $L_* (g,X)$ is that they enable the computation of the 
Goresky-MacPherson-Siegel class $L_* (X/G)$ of the orbit space $X/G$ via an
averaging formula. For $L_* (X/G)$ to be defined, one must first know
that the orbit space is also a Witt pseudomanifold.
We prove this in Corollary \ref{cor.orbitspaceiswitt}. Let us outline the strategy.
Like Zagier, we work with triangulated spaces and simplicial actions.
Foundational material on such actions is reviewed in Section \ref{sec.simplicialactions}.
For example, subanalytic proper actions admit a $G$-equivariant triangulation,
\cite[Prop. 6.7]{blp}; see Examples \ref{exples.subanalyticsimplicial}.
We show that if an oriented pseudomanifold $X$ is equipped with an oriented action
of a finite group $G$, then $X/G$ is again an oriented pseudomanifold
(Theorem \ref{thm.orbitspacepseudomanifold}).
It is harder to establish the Witt condition for $X/G$.
The idea is to proceed as in a transfer-based proof of the classical
Conner conjecture, which asserts that the orbit space
of a finite group action on a $\rat$-acyclic space is again
$\rat$-acyclic. Such arguments are for instance used to show
that the orbit space of a finite group action on a rational homology manifold
is again a rational homology manifold (Bredon \cite[V \S 19]{bredonsheaftheory}).
In fact, the Conner conjecture holds more generally
for finite ramified coverings $\pi: \wY \to Y$.
For such coverings, we construct an intersection homology transfer
$\pi_!: IH^\bp_* (Y;\intg) \to IH^\bp_* (\widetilde{Y}; \intg)$
for every $\bp$.
Such transfers appeared already in the work of Goresky and MacPherson
on the Lefschetz fixed point theorem for intersection homology, \cite{gmlefschetz}.
Our construction rests on two pillars:
an elegant formalization of ramified coverings
due to Larry Smith \cite{smithlarry} on the one hand, and on
Gajer's intersection Dold-Thom theorem (\cite{gajer}, corrected in \cite{gajercorrection})
on the other.
Smith's formalization involves symmetric products and thus, ultimately,
the free abelian topological group $AG(\wY)$ generated by the points
of $\wY$.
The intersection Dold-Thom theorem provides natural isomorphisms
$IH^\bp_* (Y;\intg) \cong I^\bp \pi_* (AG(Y))$
for connected $Y$, where
$I^\bp \pi_* (-)$ denotes the intersection homotopy groups
of a filtered space. The relationship of intersection homotopy to 
intersection homology groups
is clarified by Chataur, Saralegi-Aranguren and Tanr\'e in \cite{chataurst}.
Gajer's simplicial set can also be understood in terms of homotopy
theoretic truncations of the Postnikov tower of the links in a stratified space,
as has been shown by Chataur, Saralegi-Aranguren and Tanr\'e 
in \cite{chataursttruncations}.
Note that the free abelian topological group must be endowed with a filtration.
For this, one uses the Lawson filtration (Section \ref{sec.freeabtopgroups}).
Since the ramified covering $\pi: \wY \to Y$ is placid, it induces a
homomorphism $\pi_*: IH^\bp_* (\wY) \to IH^\bp_* (Y)$.
We construct a continuous group homomorphism
$\tau: AG(Y) \to AG(\wY)$ and prove that it is placid with respect to the
Lawson filtrations.
The desired transfer $\pi_!$ on intersection homology is then induced by
$\tau$ under the intersection Dold-Thom theorem.
There are of course other ways to define such a transfer, but
from this homotopy theoretic perspective, it is then particularly straightforward
to deduce up-down and down-up formulae.
That is, the composition $\pi_* \circ \pi_!$ is multiplication by the degree of
the covering (Proposition \ref{prop.updownmultbydegree}), 
and if $\pi: \wY=X \to X/G=Y$ is
an orbit projection, structured as a $|G|$-fold ramified covering, then
$\pi_! \circ \pi_* = \sum_{g\in G} g_*$ (Proposition \ref{prop.downupsumofgstar}).
The up-down property, in conjunction with regular neighborhood theory
and intersection homology properties of suspensions,
is then used in establishing the Witt condition for 
the base space $Y$ of a ramified covering 
(Theorem \ref{thm.baseoframifiedcoveriswitt}).
In summary, we obtain for orbit spaces:

\begin{thm*} 
Let $G$ be a finite group and $X$ an oriented Witt pseudomanifold
upon which $G$ acts by orientation preserving simplicial maps.
Then the orbit space $X/G$ is an oriented Witt pseudomanifold.
\end{thm*}

(This is Corollary \ref{cor.orbitspaceiswitt} in the main text.)
The equivariant $L$-class $L_* (g,X)$ is then constructed in
Section \ref{sec.equivariantzagierlclass}.
The method is based on Zagier's, which consists of
taking equivariant $g$-signatures $\Sign (g,X)$ of
transverse point-preimages of $G$-invariant maps from $X$
to spheres. As in \cite{atiyahsinger},
the equivariant signatures $\Sign (g,X)$
are complex numbers obtained by evaluating representation theoretic
$G$-signatures $\Sign (G,X) \in R(G)$.
For Witt spaces, however, we use the $G$-invariant intersection
form on middle intersection homology. The equivariant signature
of Witt pseudomanifolds is discussed in Section \ref{sec.gsignatures}; 
see also \cite[Section 5]{blp}.
We prove in particular that the signature of the orbit space is given in terms 
of equivariant signatures by
\[ \Sign (X/G) = \frac{1}{|G|} \sum_{g\in G} \Sign (g,X), \]
see Proposition \ref{prop.signatureoforbitspace}.

Using the equivariant $L$-classes $L_* (g,X)$, one can then
calculate the Goresky-MacPherson-Siegel $L$-class of the orbit
space as follows (Theorem \ref{thm.lclassorbitspace}).

\begin{thm*} 
Let $G$ be a finite group and $X$ an oriented closed Witt pseudomanifold
upon which $G$ acts simplicially, preserving the orientation.
Then $X/G$ is a Witt pseudomanifold and its
Goresky-MacPherson-Siegel $L$-class $L_* (X/G)$ is related to the
equivariant $L$-classes by 
\begin{equation} \label{equ.laverage}
L_* (X/G) 
    = \frac{1}{|G|} \sum_{g\in G} \pi_* L_* (g,X), 
\end{equation}
where $\pi: X \to X/G$ is the orbit projection.
\end{thm*}

For a smooth oriented closed $G$-manifold $M$, endowed
with a $G$-invariant Riemannian metric, the signature operator
$D^\sig$ is a Dirac-type operator commuting with the action of $G$.
It determines a class $[D^\sig] \in K^G_j (M),$ $j \equiv \dim M (2),$
in equivariant analytic $K$-homology
which is independent of the choice of metric.
By the Hodge theorem, $\Sign (G,M)$ is equal to the equivariant index
of the signature operator,
$\Sign (G,M) = \operatorname{ind}_G (D^{\sig, +}) \in R(G)$.
Let $X$ be a Witt $G$-pseudomanifold, equipped with an equivariant
Thom-Mather stratification and a $G$-invariant wedge metric.
As a consequence of Cheeger's Hodge-theorem
on Witt spaces  (\cite{cheeger}),
$\Sign (G,X)$ is again equal to the equivariant index of the signature 
operator $D^\sig$.
In \cite{blp}, Leichtnam, Piazza and the author use the associated
equivariant analytic $K$-homology class
$[D^\sig] \in K^G_j (X)$
to extend the fundamental result of Atiyah-Segal-Singer
from smooth $G$-manifolds to Witt $G$-pseudomanifolds.
The characteristic class formula we obtain there applies in situations
where the fixed point sets $X^g$ admit $G$-tubular neighborhoods
that have the structure of $G$-vector bundles.
The formula then expresses $\Sign (g,X)$
in terms of the Goresky-MacPherson-Siegel $L$-class of 
the Witt pseudomanifold $X^g$ and in terms 
of characteristic classes of the 
normal bundle. The latter are those characteristic classes used by Atiyah-Segal-Singer.
The proof uses Kasparov's bivariant KK-theory in an essential way, together
with the localization theorem.
In the present paper, we do not make assumptions on the normal topology
of fixed point sets, and characteristic classes of normal bundles are thus not
explicitly considered.
An investigation of the relationship of $L_* (g,X)$ to
$[D^\sig] \in K^G_j (X)$ using a suitable Chern character is beyond the
scope of this paper.

For a finite group $G$ acting continuously on a topological Witt space $X$
(satisfying weak regularity properties on the fixed point sets),
Cappell, Shaneson and Weinberger indicated the construction of a $G$-equivariant
class 
$\Delta^G (X) \in \KO^G_* (X) \otimes \intg [\smlhf]$ 
and a corresponding $G$-signature theorem in \cite{csw}.
Free group actions on PL Witt spaces were considered by Curran in \cite{curran},
and are here considered in Section \ref{sec.freeactions},
where we show that for such actions $L_* (g,X) =0$
for $g\not= 1$ (Theorem \ref{thm.freeaction}).

For singular complex algebraic varieties, homological Hirzebruch characteristic
classes have been defined by
Brasselet, Sch\"urmann and Yokura in \cite{bsy}.
Equivariant generalizations
$T_{y*} (X;g) \in H^\BM_{2*} (X^g) \otimes \cplx [y]$
have been constructed  
by Cappell, Maxim, Sch\"urmann and Shaneson in \cite{cmssequivcharcl}
for finite groups $G$
that act on quasi-projective varieties $X$ by algebraic automorphisms.
These classes satisfy the analog of our formula (\ref{equ.laverage}), i.e. their average
over all group elements computes $T_{y*} (X/G)$
(\cite[p. 1726, Thm. 1.1]{cmssequivcharcl}).
(Note that $X/G$ is again a quasi-projective variety.)

In \cite{banaglborelequivlclass},
we constructed a $G$-equivariant $L$-class
$L^G_* (X) \in H^G_* (X;\rat)$ in equivariant rational homology for
compact, oriented Whitney stratified Witt spaces $X \subset M$ that are
invariant under the oriented smooth action of a compact Lie group
$G$ on an ambient smooth manifold $M$. 
If $X=M^m$ is smooth, then equivariant Poincar\'e duality is an
isomorphism $H^{m-i}_G (X) \cong H^G_i (X),$ where
equivariant cohomology $H^*_G (-)$
means ordinary cohomology of the Borel space $(-)_G = EG\times_G (-)$.
In this case, $L^G_* (X)$ is the equivariant Poincar\'e dual
of the usual equivariant cohomology class $L^*_G (X) = L^* ((TX)_G)$,
where $(TX)_G$ is the homotopy quotient vector bundle of the equivariant tangent 
bundle $TX$. See L. Tu \cite[Chapter 30.3]{tuintrolectequivcoh} for such equivariant
characteristic classes.
The present paper is completely independent of \cite{banaglborelequivlclass}.
The class $L^G_*$ of \cite{banaglborelequivlclass}
is consistent with the equivariant viewpoint
of Ohmoto \cite{ohmoto1}, where equivariant Chern-Schwartz-MacPherson classes
were developed for algebraic actions of a complex reductive linear algebraic group 
on a possibly singular complex algebraic variety $X$.
It is also consistent with the construction of equivariant Todd classes by
Brylinski and Zhang in \cite{brylinskizhang}.\\

\textbf{Acknowledgments.}
We thank Eric Leichtnam and Paolo Piazza for many helpful discussions.
In particular, Proposition \ref{prop.wittgbordisminvariance} on 
the equivariant Witt bordism invariance of equivariant signatures for
singular spaces was developed jointly with them. We would like to thank
the Institut de Math\'ematiques de Jussieu - Paris Rive Gauche
and Universit\'a di Roma Sapienza for their hospitality during several visits.
We express our gratitude to David Chataur for a discussion on intersection
homotopy groups during a visit to the  
Université de Picardie Jules Verne in Amiens.
This work was funded in part by a research grant of the
 Deutsche Forschungsgemeinschaft (DFG, German Research Foundation)
 -- Projektnummer 495696766, and
by INdAM, Istituto Nazionale di Alta Matematica.

\section{Filtered Spaces and Placid Maps}
\label{sec.filteredspaces}

Intersection homology and intersection homotopy groups are
defined for filtered spaces. We thus review basic notions
related to filtered spaces with a particular emphasis on
stratified, placid, and cofiltered maps.
We establish that
a finite-to-one surjective simplicial map between triangulated 
finite-dimensional polyhedra
is placid with respect to the simplicial filtrations
(Lemma \ref{lem.finsurjsimpmapiscofiltered}).

\begin{defn}
Let $X$ be a nonempty topological space. 
A \emph{filtration $\Xa$ of $X$ of formal dimension $n$} is a finite
nested sequence
\[ X_{-1} = \varnothing \subset X_0 \subset X_1 \subset
\cdots \subset X_{n-2} \subset X_{n-1} \subset X_n =X \]
of subspaces, $X_{n-1} \not= X$.
The pair $(X,\Xa)$ is then called a \emph{filtered space}.
The nonempty connected components $S$ of the differences $X_i - X_{i-1}$ are called
the \emph{strata of $\Xa$ of formal dimension $i$}.
The \emph{formal codimension of $S$ in $X$} is
$\codim S := n-i$. It is convenient to set $X^i := X_{n-i}$.
\end{defn}
We do not require the subspaces $X_i$ to be closed in $X$.
The above notion of formal dimension need not be related to any
notion of topological dimension and will indeed later be applied to
spaces of infinite topological dimension. Although the formal dimension
is a property of $\Xa$ and not of the space $X$, it is nevertheless
often convenient to write $\dim X=n$ when the filtration is understood.
Base points $x_0 \in X$ of pointed filtered spaces are always
chosen to lie in $X - X_{n-1}$.
Let $\Sa (\Xa)$ denote the set of strata of a filtration $\Xa$.

\begin{lemma} \label{lem.stratumnontrivintiscontained}
Let $(X,\Xa)$ be a filtered space.
If $S \in \Sa (\Xa)$ is a stratum and $i$ an index such that
$S \cap X_i \not= \varnothing$, then
$S \subset X_i$.
\end{lemma}

\begin{lemma} \label{lem.stratumcontainedimpliescodiminequ}
Let $(X,\Xa)$ be a filtered space of formal dimension $n$ and
let $S \in \Sa (\Xa)$ be a stratum. Then 
$S\subset X_{n-k}$ if and only if $\codim S \geq k$.
\end{lemma}

\begin{defn} \label{def.stratifiedmap}
A \emph{stratified map} $f:(X,\Xa) \to (Y,\Ya)$ between filtered spaces
is a continuous map $f:X\to Y$ such that for each stratum $S$ of $\Xa$
there exists a stratum $T$ of $\Ya$ with 
$f(S) \subset T.$
\end{defn} 
Note that since $f(S)$ is nonempty and connected, the
stratum $T$ is uniquely determined by $S$ and $f$. We shall
also use the notation $f_* (S) := T$.
The above notion of stratified map agrees with Friedman's usage
of this term in \cite[Def. 2.9.1]{friedmanihbook}.

\begin{lemma} \label{lem.compstratmapsisstratmap}
The composition $g\circ f$ of stratified maps $f,g$ is again a stratified map,
and $(g\circ f)_* = g_* \circ f_*$.
\end{lemma}

\begin{defn}
A stratified map $f:(X,\Xa) \to (Y,\Ya)$ between filtered spaces
is called \emph{equidimensionally stratified}, if 
$f(X_i - X_{i-1}) \subset Y_i - Y_{i-1}$ for every $i=0,\ldots, n=\dim X$.
\end{defn}

\begin{lemma} \label{lem.compequstratmapsisequstratmap}
The composition of equidimensionally stratified maps
is again an equidimensionally stratified map.
\end{lemma}

\begin{defn} \label{def.placidmap}
A stratified map $f:(X,\Xa) \to (Y,\Ya)$ between filtered spaces
is called \emph{placid} if 
\[ \codim f_* (S) \leq \codim S \]
for every stratum $S\in \Sa (\Xa)$.
\end{defn}
This agrees with Friedman's usage of this term (\cite[Example 4.1.5]{friedmanihbook}).
We note parenthetically that what we
call a placid map is called a stratified map in \cite{chataurst}.

\begin{lemma} \label{lem.compplacidmapsisplacidmap}
The composition of placid maps is again a placid map.
\end{lemma}

\begin{defn} \label{def.gajercofilteredmap}
(Gajer \cite{gajer}.)
Let $(X,\Xa)$ and  $(Y,\Ya)$ be filtered spaces.
A continuous map $f:X \to Y$ is called 
\emph{cofiltered} with respect to $\Xa$ and $\Ya$ if
$f^{-1} (Y^k) \subset X^k$
for all $k\geq 0$.
\end{defn}

\begin{lemma} \label{lem.compcofilteredmapsiscofiltered}
The composition of cofiltered maps is cofiltered.
\end{lemma}

Let us clarify the relationship between the concepts of
stratified, placid and cofiltered maps. Simple examples show that
a cofiltered map need not be stratified.
Lemma \ref{lem.stratumcontainedimpliescodiminequ}
is helpful in proving the next lemma.
\begin{lemma} \label{lem.stratifiedandcofilteredisplacid}
A stratified cofiltered map is placid.
\end{lemma}

\begin{lemma}
Placid maps are cofiltered.	
\end{lemma}	
\begin{proof}
Let $f:(X,\Xa) \to (Y,\Ya)$ be a placid map and suppose that
$x\in f^{-1} (Y^k)$ is any point.
Let $S$ be the unique stratum of $\Xa$ that contains $x$.
Since $f$ is a stratified map, there is a unique stratum 
$T$ of $\Ya$ with $f(S) \subset T$.
We know that $f(x) \in Y^k = Y_{m-k}$, where $m=\dim Y$.
In addition, $f(x) \in f(S) \subset T$.
Thus, $f(x) \in T\cap Y_{m-k}$.
By Lemma \ref{lem.stratumnontrivintiscontained}, $T \subset Y_{m-k}$, and
by Lemma \ref{lem.stratumcontainedimpliescodiminequ}, $\codim T \geq k$.
Since $f$ is placid, $\codim T \leq \codim S$.
Therefore, $\codim S \geq k$.
By Lemma \ref{lem.stratumcontainedimpliescodiminequ},
$S\subset X_{n-k}$, where $n = \dim X$.
It follows that $x\in S \subset X_{n-k} = X^k$.
\end{proof}
Together, the above two lemmas show that a continuous map is
placid if and only if it is stratified and cofiltered. 

All of our perversity functions $\bp$ will be Goresky-MacPherson type perversities
(\cite{gmih1}). 
In particular, they only depend on the codimension of a stratum, i.e.
$\bp (S) = \bp (\codim S)$ holds on every stratum $S$;
we will use the same symbol $\bp$ to denote the function on strata
and the function on codimensions.
A placid map $f$ is ``$(\bp, \bp)$-stratified'' for every perversity $\bp$, that is,
$\bp (S) - \codim S \leq 
\bp (f_* (S)) - \codim f_* (S)$
for all strata $S \in \Sa (\Xa)$. This is a consequence of 
the Goresky-MacPherson growth condition on perversities.
See \cite[p. 137, Example 4.1.5]{friedmanihbook}.

\begin{lemma} \label{lem.finsurjsimpmapiscofiltered}
A finite-to-one surjective simplicial map between triangulated 
finite-dimensional polyhedra
is placid with respect to the simplicial filtrations.
\end{lemma}
\begin{proof}
Let $K$ and $L$ be finite-dimensional simplicial complexes
with associated polyhedra $X=|K|$ and $Y=|L|$.
The simplicial filtrations $\Xa, \Ya$ are given by the polyhedra
$X_i = |K_i|,$ $Y_i = |L_i|$ of the $i$-dimensional
simplicial skeleta $K_i \subset K,$ $L_i \subset L$.
The strata $S$ of $\Xa$ are given by the interiors
$S=\Delta^\circ$ of the simplices $\Delta \in K$.
Similarly for the strata $T$ of $\Ya$.
Formal dimensions and codimensions agree here with topological
dimensions and codimensions.
Let $n=\dim X$ and $m=\dim Y$.

Suppose that $f: K \to L$ is a simplicial map which is finite-to-one
and surjective.
In order to prove that $f: (X,\Xa) \to (Y,\Ya)$ is placid, we need
to establish first that $f$ is a stratified map.
Let $S = \Delta^\circ$ be any stratum of $\Xa$.
Since $f$ is simplicial, the images of the vertices of $\Delta$
span a simplex $f\Delta$ of $L$.
Then $T := (f\Delta)^\circ$ is a stratum of 
$\Ya$ with $f(S) \subset T$. This shows that $f$ is a stratified map. 

It remains to prove that $\codim f_* (S) \leq \codim S$ for every
stratum $S\in \Sa (\Xa)$.
The key point is that the finiteness of $f$ ensures that $f$ cannot decrease 
the dimension of simplices.
Indeed, if $f$ did decrease the dimension of some simplex, then
the linearity of $f$ on that simplex together with the pigeonhole principle
would imply that $f$ is constant on an edge, contradicting the 
finiteness of $f$.
(The simplexwise linearity of $f$ implies of course that $f$
cannot increase the dimension of simplices. Thus $f$ preserves
the dimension of simplices.)

We observe next that $n=m$, i.e. $X$ and $Y$ have the same dimension:
Since $f$ does not decrease the dimension of simplices, the
inequality $n\leq m$ holds.
The surjectivity of $f$ implies that $n\geq m$, by considering a point $y$ in 
the interior of an $m$-simplex of $L$, and then selecting a point $x$
with $f(x)=y$.

Let $S = \Delta^\circ$ be a stratum of $\Xa$.
We are now in a position to prove that 
$\codim f_* (S) \leq \codim S$.
Since $f$ was already shown to be stratified, there does indeed exist a
unique stratum $T=f_* (S)$ of $\Ya$ such that $f(S) \subset T$.
As we have seen, this stratum is the interior $T = (f\Delta)^\circ$
of the simplex $f\Delta$ of $L$ which is spanned by the images of the
vertices of $\Delta$.
Since $f$ does not decrease the dimension of simplices, we have
$\dim \Delta \leq \dim f\Delta$
and consequently
\begin{align*}
\codim T 
&= m - \dim T 
 = n - \dim f\Delta \leq n - \dim \Delta \\
&= n - \dim S 
 = \codim S.
\end{align*}
Thus $f$ is placid with respect to the simplicial stratifications.
\end{proof}
The simplicial inclusion of a vertex in an interval, which is not placid,
illustrates the necessity of the surjectivity requirement in the lemma.
Placid maps $f: (X,\Xa) \to (Y,\Ya)$ induce homomorphisms
$f_*: IH^{\bar{p}}_* (X,\Xa) \to IH^{\bar{p}}_* (Y,\Ya)$ on intersection homology
for every $\bar{p},$ see e.g. \cite[Prop. 4.1.6, p. 138]{friedmanihbook}.
Thus, by Lemma \ref{lem.finsurjsimpmapiscofiltered},
a finite surjective simplicial map $\pi: \widetilde{Y} \to Y$
(between simplicially stratified complexes),
induces a homomorphism
\[ \pi_*: IH^{\bar{p}}_* (\widetilde{Y}) \longrightarrow IH^{\bar{p}}_* (Y).  \]

\section{Free Abelian Topological Groups generated by Polyhedra}
\label{sec.freeabtopgroups}

The Dold-Thom theorem and its intersection version involve the
free abelian topological group $AG(Y)$ generated by the points of a space $Y$.
We review here some aspects of such groups, focusing eventually
on filtrations on them, and cofiltered maps between them.
Let $Y$ be a compactly generated topological space, for example a CW complex.
For every positive integer $d$, let $SP^d (Y)$ denote the
$d$-fold symmetric product 
$SP^d (Y) = Y^d / \Sigma_d,$ where the symmetric group
$\Sigma_d$ on $d$ letters acts on the $d$-fold product 
$Y^d = Y\times \cdots \times Y$ by permuting the coordinates.
Algebraically, $AG (Y) = \intg [Y]$
is the free abelian group on the underlying set of $Y$.
For every $d$, there is an obvious set-theoretic injection
$SP^d (Y) \to AG (Y)$ given by viewing points of the symmetric product,
which are just finite subsets of points of $X$ with multiplicities
adding up to $d$, as linear combinations
with those multiplicities as coefficients.
One topologizes $AG (Y)$ as the quotient space of
$\bigsqcup_{n,m} SP^n (Y) \times SP^m (Y)$
under the map
$\bigsqcup_{n,m} SP^n (Y) \times SP^m (Y) \to AG (Y)$
given by $(\phi, \psi) \mapsto \phi-\psi$.
Then $AG (Y)$ is a topological group and a compactly generated space.
A good reference is Lima-Filho \cite[Section 2.1]{limafilhogspaces}
(where our $AG(Y)$ is denoted by $\Za (Y)$).
The elements of $AG(Y)$ are represented by words $w$ which are finite formal sums
$w=\sum n_i y_i$, $n_i \in \intg,$ $y_i \in Y$.
In general, a given element has many different such representations.
To a word $w$ one can associated the function $\phi: Y \to \intg$
given by
$\phi (y) = \sum_{\{ i ~| y_i = y \}} n_i.$
Different words representing the same element of $AG(Y)$ yield the same function $\phi$.
Thus the elements of $AG(Y)$ are functions $\phi: Y \to \intg$ which vanish
at all but finitely many points of $Y$.
The \emph{support} of $\phi$ is the finite subset $\supp (\phi)$ of $Y$ consisting of all
points $y$ such that $\phi (y) \not= 0$.
To $\phi$, we can associate a word $w$ by
$w = \sum_{y\in \supp (\phi)} \phi (y) y.$
If a word is of this form, we call it \emph{reduced}.
That is, a word $w=\sum n_i y_i$ is reduced if and only if
$y_i \not= y_j$ whenever $i\not= j$ and $n_i \not= 0$ for all $i$.
The zero element $\phi =0$ is represented by the empty word.
If $w$ is any word (not necessarily reduced), its support is defined to 
be the support of the function $\phi$ associated to $w$.
If $w=\sum n_i y_i$ is a reduced word, then
$\supp (w) = \{ y_i ~|~ i \}.$

Now let $Y$ be the polyhedron of a countable simplicial
complex, not necessarily connected, and
let $Y^+ := Y \sqcup \{ * \}$
be the disjoint topological union of $Y$ with an isolated point $* \not\in Y$.
The pointed space $(Y^+,*)$ has an associated abelian topological group
$AG (Y^+,*)$
constructed by Dold-Thom in \cite{doldthom}. 
Its underlying algebraic group is $\intg [Y^+ - *] = \intg [Y]$.
Their theorem \cite[p. 274, Satz 6.10 I for $Y$]{doldthom} provides an isomorphism 
$H_* (Y;\intg) = \widetilde{H}_* (Y^+;\intg) \cong
   \pi_* (AG(Y^+, *), 0).$
The polyhedron $Y$ need not be connected here.
There is a continuous isomorphism $AG (Y^+ ,*) \cong AG (Y)$
of topological groups (as follows from 
\cite[p. 570 b; and p. 571, Lemma 2.2 a]{limafilhogspaces}). Thus
\[ H_* (Y;\intg) \cong \pi_* (AG(Y),0) \]
as stated in \cite{gajer}.
As $Y$ is the polyhedron of a countable simplicial complex,
$AG(Y^+,*)$ and $AG(Y)$ have the structure of a CW complex.
From now on, the term polyhedron will mean the polyhedron of a 
countable simplicial complex.
There are embeddings $Y=SP^1 (Y) \subset SP^d (Y) \hookrightarrow AG (Y)$
for every $d$. The following universal property is implied by
\cite[p. 570, d]{limafilhogspaces}.

\begin{prop} \label{prop.markovuniversal}
(Lima-Filho.)
For a compactly generated space $Y$, 
the map $j: Y \to AG (Y)$ satisfies the following
universal property: 
Let $G$ be an abelian topological compactly generated group
and let $f:Y \to G$ be a continuous map. Then there exists a unique
topological group homomorphism $F: AG (Y) \to G$ such that $F \circ j =f$.
\end{prop}

A continuous map $f:Y \to Y'$ between polyhedra as above induces
a well-defined continuous homomorphism
$AG(f): AG(Y) \to AG(Y')$
by setting
$AG(f)(\sum n_i y_i) := \sum n_i f(y_i),$
(\cite[p. 961]{gajer}), that is, 
$AG(f)(\phi)(y') = \sum_{y\in f^{-1} (y')} \phi (y).$
In this way, $AG (-)$ becomes a covariant functor
(\cite[p. 570, Properties 2.1 a]{limafilhogspaces},
\cite[6.6]{mccord}).\\

Let $Y$ be a space which is the polyhedron of a countable simplicial complex.
A filtration $\Ya$ of $Y$ induces a filtration $AG(\Ya),$ the 
\emph{Lawson filtration}, of
$AG (Y)$ as follows: The formal dimension $m$ of $AG(\Ya)$ is declared 
to be equal to the formal dimension of $\Ya$, and the filtration subspaces
are given for $k>0$ by
\[
AG(Y)_{m-k} = AG(Y)^k := \{ \phi \in AG(Y) ~|~
             \supp (\phi) \cap Y^k \not= \varnothing \}.
\]
If $\phi =0$ is the zero element of $AG(Y)$, then its support is
empty. Thus is does not belong to any $AG(Y)_{m-k},$ $k>0$.
For $k=0$, we define
\[
AG(Y)_m = AG(Y)^0 := \{ \phi \in AG(Y) ~|~
             \supp (\phi) \cap Y^0 \not= \varnothing \} \cup \{ 0 \}
         = AG(Y).
\]
If $k>0$ and
$\supp (\phi) \cap Y_{m-k} \not= \varnothing$, then
$\supp (\phi) \cap Y_{m-k+1} \not= \varnothing$. This shows that
\[ AG(Y)^k = AG(Y)_{m-k} \subset AG(Y)_{m-k+1} = AG(Y)^{k-1}. \]
The next lemma follows by a straightforward verification from definitions.
\begin{lemma} \label{lem.canonemyagyiseqdstratandcof}
Let $(Y,\Ya)$ be a filtered polyhedron.
The canonical embedding
\[ i:(Y,\Ya) \hookrightarrow (AG(Y), AG(\Ya)) \]
is equidimensionally stratified and cofiltered.
\end{lemma}

Let us investigate when the extension of an equidimensionally
stratified and cofiltered map $Y\to AG(X)$ to a continuous homomorphism
$AG(Y) \to AG (X)$ is again equidimensionally stratified and cofiltered.
This will be so under the following condition.
\begin{defn} \label{def.separatedmapintoag}
Let $X$ be a polyhedron and $Y$ a topological space.
A map $f:Y \to AG(X)$ is called \emph{separated},
if for all $y,y' \in Y$ with $y\not= y'$, we have
$\supp (f(y)) \cap \supp (f(y')) = \varnothing.$
\end{defn}

\begin{lemma} \label{lem.extensionisequistrat}
Let $(X,\Xa)$ be a filtered polyhedron and $(Y,\Ya)$ a filtered compactly
generated space such that $\dim X = \dim Y$.
Let $f: Y \to AG(X)$ be a separated continuous map.
Let $F: AG(Y) \to AG(X)$ be the unique extension of
$f$ to a continuous group homomorphism.
If $f$ is equidimensionally stratified and cofiltered
(with respect to $\Ya$ and the
Lawson filtration on $AG(X)$ induced by $\Xa$), then
$F$ is equidimensionally stratified and cofiltered
(with respect to the Lawson filtrations).
\end{lemma}
\begin{proof}
Set $n := \dim X = \dim Y$. Then by definition of the Lawson filtration,
$n = \dim AG(X)=\dim AG(Y).$
Therefore, given a codimension $k\geq 0,$
we must show that 
\[ F(AG(Y)^k - AG(Y)^{k+1}) \subset AG(X)^k - AG(X)^{k+1}.\]
Since $f$ is equidimensionally stratified
and $\dim Y= \dim AG(X)$, we have
\begin{equation} \label{equ.fykinagc}
f(Y^k - Y^{k+1}) \subset AG(X)^k - AG(X)^{k+1}. 
\end{equation}
Let 
$w=\sum n_i y_i \in AG(Y)^k - AG(Y)^{k+1}$
be a reduced word, $\supp (w) = \{ y_i ~|~ i \}$.
Thus
\[ \supp (w) \cap Y^k \not= \varnothing \text{ and }
   \supp (w) \cap Y^{k+1} = \varnothing. \]
Consequently, there exists an index $i_0$ such that
\begin{equation} \label{equ.yi0inyk}
y_{i_0} \in Y^k - Y^{k+1}. 
\end{equation}
We will show that
$\supp (F(w)) \cap X^k \not= \varnothing$ and
$\supp (F(w)) \cap X^{k+1} = \varnothing,$
which will then place $F(w)$ in $AG(X)^k - AG(X)^{k+1}$.
In order to do this, we shall determine the support of $F(w)$.
The extension $F,$ being a group homomorphism, is given by
$F(w) = \sum n_i f(y_i).$
For every $i$, let 
$f(y_i) = \sum_j m_{ij} x_{ij}$
be the representation by a reduced word, $m_{ij} \in \intg - \{ 0\},$
$x_{ij} \in X$. Since for every $i$, the word is reduced we know
that $x_{ij} \not= x_{ij'}$ if $j\not= j'$.
Thus the support of $f(y_i)$ is
$\supp (f(y_i)) = \{ x_{ij} ~|~ j \}.$
The image of $w$ can be expanded as
$F(w) = \sum_i \sum_j n_i m_{ij} x_{ij}.$
We claim that this word is reduced.
Indeed, if $i\not= i'$, then $y_i \not= y_{i'}$ as $w$ is reduced.
Therefore, since $f$ is separated, 
\[ \{ x_{ij} ~|~ j \} \cap \{ x_{i' j} ~|~ j \}
  = \supp (f(y_i)) \cap \supp (f(y_{i'})) = \varnothing.  \]
We conclude that 
$x_{ij} \not= x_{i' j'}$ whenever $(i,j) \not= (i',j')$.
Furthermore, $n_i m_{ij} \not= 0$ for all $(i,j)$.
This shows that $\sum_{i,j} n_i m_{ij} x_{ij}$ is a reduced word,
which implies that
$\supp (F(w)) = \{ x_{ij} ~|~ i,j \}.$
By (\ref{equ.fykinagc}) and (\ref{equ.yi0inyk}),
$f(y_{i_0}) \in AG(X)^k - AG(X)^{k+1}.$
So
$\supp (f(y_{i_0})) \cap X^k \not= \varnothing$ and
$\supp (f(y_{i_0})) \cap X^{k+1} = \varnothing.$
Hence, there exists a point
$x_{i_0 j_0} \in \supp (f(y_{i_0})) \cap X^k$.
Since
\[ \supp (f(y_{i_0})) = \{ x_{i_0 j} ~|~ j \}
  \subset \{ x_{ij} ~|~ i,j \} = \supp F(w), \]
the point $x_{i_0 j_0}$ is in $\supp F(w) \cap X^k$, which shows that
$\supp (F(w)) \cap X^k \not= \varnothing$.

We will verify next that $\supp (F(w)) \cap X^{k+1} = \varnothing$.
This will use the assumption that $f$ is cofiltered.
We proceed by contradiction:
Suppose that $\supp (F(w)) \cap X^{k+1} \not= \varnothing$.
Then there exists a pair $(i_0, j_0)$ such that  
$x_{i_0 j_0} \in X^{k+1}$.
Since this point is in the support of the corresponding $f(y_{i_0}),$
it follows that $f(y_{i_0}) \in AG(X)^{k+1}$.
As $f$ is cofiltered, the inclusion
$f^{-1} (AG(X)^{k+1}) \subset Y^{k+1}$ holds,
which implies that 
$y_{i_0} \in Y^{k+1}$.
This contradicts $\supp (w) \cap Y^{k+1} = \varnothing$.
Therefore, $\supp (F(w)) \cap X^{k+1} = \varnothing$ as claimed.
We have shown that $F$ is equidimensionally stratified.

It remains to be shown that $F$ is cofiltered.
We must thus establish the inclusion
\[ F^{-1} (AG(X)^k) \subset AG(Y)^k \] 
for every $k\geq 0$.
Given any element $w\in F^{-1} (AG(X)^k)$,
let $q\geq 0$ be the unique index with
$w\in AG(Y)^q - AG(Y)^{q+1}$.
Then, as $F$ is already known to be equidimensionally stratified,
we have $F(w) \in AG(X)^q - AG(X)^{q+1}$.
Hence $AG(X)^k \cap (AG(X)^q - AG(X)^{q+1}) \not= \varnothing,$
which can only happen when $k\leq q$.
It follows that $w \in AG(Y)^q \subset AG(Y)^k,$ as was to be shown.
\end{proof}

The following lemma will be of use in the context of a group
$G$ acting on a filtered polyhedron $(Y,\Ya)$ by placid homeomorphisms
$g: (Y,\Ya) \to (Y,\Ya)$.
\begin{lemma} \label{lem.aggequistratcofiltered}
Let $(X,\Xa)$ and $(Y,\Ya)$ be filtered polyhedra of equal dimension
$\dim X = \dim Y$.
Let $g: (Y,\Ya) \to (X,\Xa)$ be an equidimensionally stratified
cofiltered map. If $g$ is injective, 
then the induced morphism $AG(g): AG(Y) \to AG(X)$
is equidimensionally stratified and cofiltered
(with respect to the Lawson filtrations).
\end{lemma}
\begin{proof}
Let $f: Y \to AG(X)$ be the composition of $g$ with the
canonical inclusion $i: X \hookrightarrow AG(X)$.
In view of the commutative square
\[ \xymatrix{
Y \ar@{^{(}->}[d] \ar[r]^g \ar[rd]_f & X \ar@{^{(}->}[d]^i \\
AG(Y) \ar[r]_{AG(g) = F} & AG(X),
} \]
Lemma \ref{lem.extensionisequistrat} will imply that $AG(g)=F$ is equidimensionally
stratified and cofiltered once we have shown that 
$f$ is separated, equidimensionallly stratified and cofiltered.
We shall now check these properties, beginning with the separation.
Let $y,y' \in Y$ be two points, $y\not= y'$.
The support of $f(y)$ is given by the one-point set
$\{ g(y) \}$. The injectivity of $g$ thus implies
\[ \supp (f(y)) \cap \supp (f(y')) 
  = \{ g(y) \} \cap \{ g(y') \} = \varnothing.   \]
This shows that $f$ is separated.
By Lemma \ref{lem.canonemyagyiseqdstratandcof},
the canonical embedding
$i$ is equidimensionally stratified and cofiltered.
Since $g$ is, by assumption, equidimensionally stratified and cofiltered,
the composition $f=i\circ g$ is
equidimensionally stratified and cofiltered
(Lemmas \ref{lem.compequstratmapsisequstratmap}, 
\ref{lem.compcofilteredmapsiscofiltered}).
\end{proof}

\begin{remark}
The injectivity condition in Lemma \ref{lem.aggequistratcofiltered}
is necessary. For an equidimensionally stratified and cofiltered map $g$ which is
not injective, the induced map $AG(g)$ will not generally be
equidimensionally stratified.
Consider for example the collapse map $g: Y=[0,1] \to S^1=X$ from the unit
interval to the circle with $g(0)=g(1)$. Both source and domain are
to have formal dimension $1$, and $Y_0 = \{ 0,1 \},$ $X_0 = \{ g(0)=g(1) \}$.
Then $g$ is equidimensionally stratified and cofiltered.
Consider the reduced word $w= y_0 - y_1 \in AG(Y)$, where
$y_0 =0\in Y$, $y_1 =1 \in Y$. 
Then $w \in AG(Y)^1$, but
$AG (g)(w) = g(y_1) - g(y_2) =0$ has empty support and
$AG(g)(w) \in AG (X)^0 - AG(X)^1$.
\end{remark}

The property of being cofiltered is preserved by $AG(-)$, as was
already observed by Gajer.
\begin{lemma} \label{lem.cofilteredpresbyag}
(See Gajer \cite[p. 961]{gajer}.)
Let $(X,\Xa)$ and $(Y,\Ya)$ be filtered polyhedra.
If $g: (Y,\Ya) \to (X,\Xa)$ is
cofiltered, then the induced morphism $AG(g): AG(Y) \to AG(X)$
is cofiltered (with respect to the Lawson filtrations).
\end{lemma}

\section{Intersection Simplicial Sets and Intersection Homotopy Groups}

Simplicial sets (and more generally simplicial objects) 
will be denoted by bold face letters such as $\sS$.
Let $\sS (X)$ denote the singular simplicial set of a topological 
space $X$. It satisfies the Kan extension condition and
if $x_0$ is a base point in $X$, then
\begin{equation}  \label{equ.htpygrpsofsingsimplset}
\pi_n (\sS (X), x_0) = \pi_n (X, x_0), 
\end{equation}
where $x_0$ also denotes the simplicial subset of $\sS (X)$
generated by $x_0: \Delta^0 \to X$.

\begin{prop} \label{prop.sgissimplicialgroup}
If $X=G$ is a topological group, then $\sS (G)$ is a simplicial group.
If $G$ is abelian, then $\sS (G)$ is an abelian simplicial group.
If $f:G\to H$ is a continuous homomorphism of topological groups,
then the induced simplicial map $\sS (f): \sS (G) \to \sS (H)$
is a simplicial homomorphism of simplicial groups.
\end{prop}
\begin{proof}
We must endow every $\sS (G)_n$ with a group structure such that
the face maps $d_i: \sS (G)_n \to \sS (G)_{n-1}$ and the
degeneracy maps $s_i: \sS (G)_n \to \sS (G)_{n+1}$ are group
homomorphisms.
The set $\sS (G)_n$ is the set of all continuous maps
$\Delta^n \to G$. Given two such maps $\sigma, \tau$, we
define their product $\sigma \cdot \tau$ pointwise by
$(\sigma \cdot \tau)(t) := \sigma (t) \tau (t)$ for $t\in \Delta^n$,
using the group operation on $G$.
It is then straightforward to verify the required properties.
\end{proof}

A subspace $A$ of a polyhedron $P$ is said to have
\emph{polyhedral dimension} less than or equal to $k$ if
$A$ is contained in a subpolyhedron $Q\subset P$ with $\dim Q \leq k$.
In this case, one writes $\dim A \leq k$.
An important property of polyhedral dimension is its monotonicity,
that is, if $B \subset A \subset P$, then
$\dim B \leq \dim A$ in the sense that for any $k$ with $\dim A \leq k$,
one has $\dim B \leq k$.
Let $(X,\Xa)$ be a filtered space and
let $\bar{p}$ be a (Goresky-MacPherson) perversity.
A continuous map $f: P \to X$ is called
\emph{$\bp$-allowable} (with respect to $\Xa$) 
if for every skeleton $X^s$ of $\Xa$
the polyhedral dimension of the preimage of the skeleton satisfies
\[ \dim f^{-1} (X^s) \leq \dim P - s + \bar{p} (s). \]
Here, the empty set is deemed to have dimension $-\infty$.
This definition is used by Gajer, \cite[p. 943]{gajer}.
If $f$ is $\bp$-allowable, then it is $\bp$-allowable in the sense
of Chataur et al.  \cite[Def. 2.12, p. 9]{chataurst}.
Let $(M,\partial M)$ be a PL manifold.
A continuous map $f:M \to X$ is called a \emph{$\bp$-map}
(with respect to $\Xa$), if both $f$ and its restriction to the
boundary $\partial M$ are $\bp$-allowable.
Thus, if $\partial M$ is empty, then a $\bp$-map is the same thing
as a $\bp$-allowable map.
A \emph{$\bp$-homotopy} is a $\bp$-map $F: M\times I \to X$.
A singular simplex $\sigma: \Delta^i \to X$ is called \emph{$\bar{p}$-full}
(Chataur-Saralegi-Tanr\'e \cite[p. 10, Def. 3.1]{chataurst})
if $\sigma$ and all of its faces
$d_{j_1} \cdots d_{j_k} (\sigma)$ are $\bar{p}$-allowable.
If $\sigma$ is $\bar{p}$-full, then every degenerate simplex
$s_{j_1} \cdots s_{j_k} (\sigma)$ is $\bar{p}$-full,
see Gajer \cite[p. 945]{gajer}.
Taking $\sIS^{\bar{p}} (X,\Xa)_i \subset \sS (X)_i$ to be the set of all
$\bar{p}$-full singular simplices $\sigma: \Delta^i \to X$,
one thus obtains a sub-simplicial set
\[ \sIS^{\bar{p}} (X,\Xa) \subset \sS (X) \]
of the singular simplicial set of $X$, the
\emph{intersection singular simplicial set} of $(X,\Xa)$.
If the perversity $\bar{p}$ is understood, we will often write
$\sIS (X,\Xa)$ for $\sIS^{\bar{p}} (X,\Xa)$, and if the filtration 
$\Xa$ on $X$ is understood we will tend to write
$\sIS (X)$ for $\sIS (X,\Xa)$.
(In \cite{chataurst}, $\sIS (X)$ is called the \emph{Gajer $\bar{p}$-space}.)
The intersection singular simplicial set $\sIS (X)$ satisfies the
Kan condition.

If the underlying topological space of
a topological group $G=X$ is equipped with a filtration $\Ga = \Xa$ of formal
dimension $n$,
and the neutral element $1\in G$ is to serve as a base point (which is typically the case),
then $1$ must be in $G-G_{n-1}$. Thus every subgroup of $G$ must intersect
$G-G_{n-1}$ nontrivially.
\begin{defn}
Let $G=X$ be a topological group. A filtration $\Ga = \Xa$ of the underlying
topological space of $G$ is called \emph{compatible with the group structure}
if $G-G_i$ is an algebraic subgroup of $G$ for every $i<n$, where $n$
is the formal dimension of $\Ga$.
\end{defn}
The examples relevant for us are given by the Lawson filtration of the free 
abelian topological group on a polyhedron:
\begin{lemma} \label{lem.lawsonfiltriscompatible}
Let $(Y,\Ya)$ be a filtered polyhedron.
Then the Lawson filtration $AG(\Ya)$ on the free abelian topological group
$AG(Y)$ is compatible with the group structure.
\end{lemma}
\begin{proof}
The proof is straightforward and mainly rests on the fact that
$\supp (\phi - \psi) \subset \supp (\phi) \cup \supp (\psi)$
for $\phi, \psi \in AG(Y)$.
\end{proof}

\begin{prop} \label{prop.isgissimplgrp}
Let $G=X$ be a topological group whose underlying topological space
is endowed with a filtration $\Ga = \Xa$.
If $\Ga$ is compatible with the group structure,
then the sub-simplicial set
$\sIS^\bp (G, \Ga) \subset \sS (G)$
is a simplicial subgroup of the simplicial group $\sS (G)$.
\end{prop}
\begin{proof}
Let $\sigma_1: \Delta^k \to G$ be the unit of $\sS (G)_k$, i.e.
the constant map $\sigma_1 (t)=1$. We claim that $\sigma_1$ is
$\bp$-full. Since the faces of $\sigma_1$ are again units,
the claim follows once we have shown that $\sigma_1$ is
$\bp$-allowable, that is,
$\dim \sigma^{-1}_1 (G^s) \leq k-s+\bp (s)$
for all $s\geq 0$. 
If $s=0$, then $G^0=G$ and $\sigma^{-1}_1 (G) = \Delta^k$.
Therefore,
$\dim \sigma^{-1}_1 (G) = k = k-0+\bp (0),$ as required.
If $s>0$, then $1\not\in G^s$, since $1 \in G^0 - G^1$.
Hence in this case, $\sigma^{-1}_1 (G^s) = \varnothing$ and
$\dim \sigma^{-1}_1 (G^s) = \dim \varnothing =
  -\infty \leq  k-s+\bp (s),$ as required.
We have shown that the unit of $\sS (G)_k$ is in $\sIS^\bp (G, \Ga)$.

Let $\sigma, \tau \in \sIS^\bp (G, \Ga)_k$ be elements. 
We have to show that their product $\sigma \cdot \tau,$ taken in the group
$\sS (G)_k$ as explained in Proposition \ref{prop.sgissimplicialgroup}, 
actually lies in $\sIS^\bp (G, \Ga)_k$.
The singular simplices $\sigma, \tau: \Delta^k \to G$
are $\bp$-full, that is, they are $\bp$-allowable and all of their
faces are also $\bp$-allowable.
Thus for all $s\geq 0$, we have
$\max (\dim \sigma^{-1} (G^s), \dim \tau^{-1} (G^s)) \leq k - s + \bp (s).$
 We claim that 
\begin{equation} \label{equ.preimageofsigmaplustau} 
(\sigma \cdot \tau)^{-1} (G^s)
  \subset \sigma^{-1} (G^s) \cup \tau^{-1} (G^s). 
\end{equation}  
Indeed, suppose that $t\in \Delta^k$ is a point with
$\sigma (t) \tau (t) = (\sigma \cdot \tau)(t) \in G^s.$
If both $\sigma (t)$ and $\tau (t)$ were in $G-G^s$,
then their product would be in $G-G^s$, since the latter
is a subgroup by compatibility of $\Ga$.
Thus at least one of $\sigma (t), \tau (t)$ must be in $G^s$.
This proves the claim (\ref{equ.preimageofsigmaplustau}).
The polyhedral dimension is monotone and satisfies
$\dim (A_1 \cup A_2) = \max (\dim A_1,~ \dim A_2),$
see also \cite[p. 9]{chataurst}.
Thus by (\ref{equ.preimageofsigmaplustau}),
\begin{align*}
\dim (\sigma \cdot \tau)^{-1} (G^s)
&\leq \dim (\sigma^{-1} (G^s) \cup \tau^{-1} (G^s)) \\
&= \max (\dim \sigma^{-1} (G^s),~ \dim \tau^{-1} (G^s)) 
    \leq k-s+ \bp (s).
\end{align*}
This shows that $\sigma \cdot \tau: \Delta^k \to G$ is $\bp$-allowable.
The $\bp$-allowability of its faces
$d_{j_1} \cdots d_{j_m} (\sigma \cdot \tau)$
follows similarly, using
that the simplicial face maps $d_j$ in $\sS (G)$ are group homomorphisms.
This shows that $\sigma \cdot \tau$ is $\bp$-full, and thus
an element of 
$\sIS^\bp (G, \Ga))_k$.
Finally, one verifies similarly that $\sIS^\bp (G, \Ga))_k$ is closed under taking
inverses.
\end{proof}

The next statement is then an immediate consequence of
Lemma \ref{lem.lawsonfiltriscompatible} and Proposition \ref{prop.isgissimplgrp}.
\begin{prop} \label{prop.isagyissimplabgrp}
Let $Y$ be a polyhedron triangulated by a countable simplicial complex
and let $\Ya$ be a filtration on $Y$.
Then the sub-simplicial set
$\sIS^\bp (AG(Y), AG(\Ya)) \subset \sS (AG(Y))$
is a simplicial abelian subgroup of the simplicial abelian group $\sS (AG(Y))$.
\end{prop}

\begin{prop} \label{prop.cofilteredinducesmaponsisp}
A cofiltered map $f: (X,\Xa) \to (Y,\Ya)$ between filtered spaces
induces a simplicial map
$\sIS^{\bar{p}} (f): \sIS^{\bar{p}} (X,\Xa) \to \sIS^{\bar{p}} (Y,\Ya).$
If $X$ and $Y$ are topological groups, and $f$ is
in addition a group homomorphism, then
$\sIS^{\bar{p}} (f)$ is a simplicial homomorphism of simplicial groups.
\end{prop}
\begin{proof}
The first statement has been observed by Gajer \cite[p. 946]{gajer};
see also Chataur et al. \cite[p. 11, Prop. 3.5]{chataurst}.
If $\sigma: \Delta^i \to X$ is a $\bp$-full singular simplex,
one takes $\sIS (f)$ to be the restriction of $\sS (f)$, that is,
$\sIS (f)(\sigma) = f \circ \sigma: \Delta^i \to Y.$
This works, as the composition
$f\circ \sigma$ is again $\bp$-full in $(Y,\Ya)$ if $f$ is cofiltered.
The map $\sIS (f)$ is simplicial, since it is the restriction of
$\sS (f): \sS (X) \to \sS (Y)$, which is simplicial.
If $X$ and $Y$ are topological groups and $f:X\to Y$
is a continuous homomorphism, then 
$\sS (f): \sS (X) \to \sS (Y)$ is a homomorphism of simplicial groups
by Proposition \ref{prop.sgissimplicialgroup}.
The map $\sIS (f)$ is the restriction of $\sS (f)$ to $\bp$-full simplices,
and thus also a group homomorphism.
Here, both $\sIS (X)$ and $\sIS (Y)$ are simplicial subgroups by
Proposition \ref{prop.isgissimplgrp}.
\end{proof}

We recall that base points $x_0$ of pointed filtered spaces
$(X,\Xa, x_0)$ are required to be in the top stratum, $x_0 \in X^0 - X^1$.
The \emph{intersection homotopy groups} of a pointed filtered space
$(X,\Xa, x_0)$ associated
to $\bp$ are given by the simplicial homotopy groups of the simplicial
set $\sIS^{\bar{p}} (X,\Xa),$
\[ I^\bp \pi_k (X,\Xa,x_0) := \pi_k (\sIS^{\bar{p}} (X,\Xa), x_0), \]
see \cite[p. 946]{gajer} and \cite[p. 16, Def. 4.1]{chataurst}.
This may be viewed as an intersection analog to (\ref{equ.htpygrpsofsingsimplset}).
Since $\sIS (X)$ satisfies the Kan condition, the elements of intersection
homotopy groups are represented by $\bp$-full singular simplices in $X$
whose boundary maps to $x_0$.
(The homology of $\sIS^{\bar{p}} (X,\Xa)$ will never be considered in this paper,
nor will we ever use the simplicial object $AG (\sIS^{\bar{p}} (X,\Xa))$.)

\begin{lemma} \label{lem.cofilteredonaginducesoninthtpygrps}
Let $f:(X,\Xa, x_0) \to (Y,\Ya, y_0)$ be a cofiltered pointed map between filtered 
pointed spaces.
Then $f$ induces a group homomorphism
$f_*: I^\bp \pi_* (X, \Xa, x_0) \to I^\bp \pi_* (Y, \Ya, y_0)$
for every $\bp$.
\end{lemma}
\begin{proof}
By Proposition \ref{prop.cofilteredinducesmaponsisp},
$f$ induces a simplicial map
$\sIS^{\bar{p}} (f): \sIS^{\bar{p}} (X,\Xa) \to \sIS^{\bar{p}} (Y,\Ya).$
This map induces the desired homomorphism $f_*$ on 
simplicial homotopy groups.
\end{proof}

Let $(Y,\Ya)$ be a filtered polyhedron.
The filtration $\Ya$ is called a \emph{PL stratification}
if all $Y_i$ are subpolyhedra, the strata are PL manifolds,
and $\Ya$ is PL locally cone-like. Examples are the skeletal
filtration of a PL triangulation.
Gajer's intersection Dold-Thom theorem \cite{gajer}, corrected in \cite{gajercorrection},
provides natural isomorphisms
\[  IH^\bp_* (Y;\intg) \cong I^\bp \pi_* (AG(Y), AG (\Ya)) \]
for connected PL stratified polyhedra $Y$.

With a view towards the up-down and down-up properties
(Propositions \ref{prop.updownmultbydegree} and \ref{prop.downupsumofgstar}),
we establish the additivity of $\pi_* \sIS^\bp$. We will focus on endomorphisms,
though many of the statements are more generally true for homomorphisms.
Let $\sA$ be a simplicial abelian group.
The set $\End (\sA)$ of endomorphisms of $\sA$ is an abelian group
by defining the sum $f+g$ of $f,g \in \End (\sA)$ to be
$(f+g)_n:\sA_n \to \sA_n,$
$(f+g)_n (a) = f_n (a) + g_n (a),~ a\in \sA_n$, 
using the group law in $\sA_n$.

\begin{lemma} \label{lem.piadditive}
Given endomorphisms $f,g \in \End (\sA)$, additivity
$\pi_* (f+g) = \pi_* (f) + \pi_* (g)$
holds in $\End (\pi_* (\sA))$.
\end{lemma}
This follows e.g. from expressing the homotopy group via the 
Dold-Kan correspondence as the homology group of the Moore chain complex
and using the fact that both the Moore and the homology functor are additive.
Let $F,G:A \to A'$ be continuous homomorphisms of abelian topological
groups. Their sum $F+G:A \to A'$, given by the composition
\[ A \stackrel{(F,G)}{\longrightarrow}
  A' \times A' \stackrel{+}{\longrightarrow} A', \]
is again a continuous homomorphism.
In this way, the set $\Hom (A,A')$ of continuous homomorphisms is an abelian group.
In particular, $\End (A)$ is an abelian group.
\begin{lemma} \label{lem.singsimplsetfunctoradditive}
Let $A$ be an abelian topological group, for example $A=AG(Y)$.
Given $F,G \in \End(A),$ the additive compatibility
$\sS (F+G) = \sS (F) + \sS (G)$
holds in $\End (\sS (A))$.
\end{lemma}
This is a straightforward verification using the various (pointwise) addition laws
recalled above, in particular using Proposition \ref{prop.sgissimplicialgroup}.

\begin{lemma} \label{lem.sumofcofilterediscofandzeroisplacid}
Let $f,g: AG(Y) \to AG(Y)$ be continuous maps. 
If $f$ and $g$ are cofiltered with respect to the Lawson filtration,
then the continuous map $f+g: AG(Y) \to AG(Y)$ given by 
$(f+g)(\phi) = f(\phi) + g(\phi)$ is also cofiltered.
The zero endomorphism $0: AG(Y) \to AG(Y)$ is placid.
\end{lemma}
\begin{proof}
As $f$ and $g$ are cofiltered, we know that
$f^{-1} (AG(Y)^k) \subset AG(Y)^k,$
$g^{-1} (AG(Y)^k) \subset AG(Y)^k.$
We have to show that
$(f+g)^{-1} (AG(Y)^k) \subset AG(Y)^k.$
Let $\phi \in AG(Y)$ be an element with $(f+g)(\phi) \in AG(Y)^k$.
By definition of the Lawson filtration, membership in $AG(Y)^k$
means that
$\supp ((f+g)(\phi)) \cap Y^k \not= \varnothing.$
Consequently, there is some point $y_0 \in Y^k$
with $f(\phi)(y_0) + g(\phi)(y_0) \not= 0$.
Hence $f(\phi)(y_0) \not= 0$
or $g(\phi)(y_0) \not= 0$.
If $f(\phi)(y_0) \not= 0$, then
$\phi \in f^{-1} (AG(Y)^k) \subset AG(Y)^k.$
Similarly, $\phi$ is also in $AG(Y)^k$ when $g(\phi)(y_0) \not= 0$.

To show that the zero endomorphism $0$ on $AG(Y)$ is placid, 
let us first verify that it is cofiltered.
We must show that $0^{-1} (AG(Y)^k) \subset AG(Y)^k$ for all $k\geq 0$.
If $k=0$, then $0\in AG(Y)^0$ and thus
$0^{-1} (AG(Y)^0) = AG(Y) = AG(Y)^0.$
If $k>0$, then $0\not\in AG(Y)^k$ and thus
\[ 0^{-1} (AG(Y)^k) = \varnothing \subset AG(Y)^k. \]
Thus $0$ is cofiltered.
Let $T$ be the connected component of $0$ in $AG(Y)_m - AG(Y)_{m-1},$
where $m$ is the formal dimension of the filtered polyhedron $Y$.
Then for every stratum $S$ of $AG(Y)$, the image $0(S)$ is contained
in $T$. This shows that $0$ is a stratified map.
By Lemma \ref{lem.stratifiedandcofilteredisplacid}, $0$ is placid.
\end{proof}

\begin{remark}
If $f,g: AG(Y)\to AG(Y)$ are
stratified with respect to the Lawson filtration, then
it is not generally true that $f+g$ is again stratified.
\end{remark}

\begin{prop} \label{prop.piisisadditive}
Let $F,G\in \End (AG(Y))$ be endomorphisms.
If $F$ and $G$ are cofiltered with respect to the Lawson filtration, then
\[ \pi_* \sIS^\bp (F+G) =  \pi_* \sIS^\bp (F) + \pi_* \sIS^\bp (G)
   \in \End (\pi_* \sIS^\bp (AG(Y))).  \]
Under the intersection Dold-Thom identification 
$IH^\bp_* (Y;\intg) = I^\bp \pi_* (AG(Y)) =
\pi_* (\sIS^\bp (AG(Y))),$ where we take $Y$ to be connected, we thus have
$(F+G)_* = F_* + G_* \in \End (IH^\bp_* (Y;\intg)).$
\end{prop}
\begin{proof}
We will write $\sIS = \sIS^\bp$. The sum $F+G$ is cofiltered
by Lemma \ref{lem.sumofcofilterediscofandzeroisplacid}.
Since $F,G$ and $F+G$ are cofiltered, the endomorphisms
$\sIS (F),$ $\sIS (G),$ $\sIS (F+G) \in \End (\sIS (AG(Y)))$
are well-defined by Proposition \ref{prop.cofilteredinducesmaponsisp}.
The simplicial abelian group $\sIS (AG(Y))$ is a simplicial subgroup
of the simplicial abelian group $\sS (AG(Y))$ by
Proposition \ref{prop.isagyissimplabgrp}.
Thus there is a commutative diagram
\[ \xymatrix{
\sIS (AG(Y)) \ar@{^{(}->}[d] \ar[r]^{\sIS (F)} & \sIS (AG(Y)) \ar@{^{(}->}[d] \\
\sS (AG(Y)) \ar[r]_{\sS (F)} & \sS (AG(Y)),
} \]
i.e. $\sIS (F)$ is the restriction of $\sS (F)$ to $\sIS (AG(Y))$.
There are similar diagrams for $\sIS (G), \sS (G)$ and for
$\sIS (F+G), \sS (F+G)$.
By Lemma \ref{lem.singsimplsetfunctoradditive}, 
$\sS (F) + \sS (G) = \sS (F+G)$.
Restricting this identity to $\sIS (AG(Y))$, we obtain
$\sIS (F) + \sIS (G) = \sIS (F+G)$.
Taking $\sA = \sIS (AG(Y))$, $f= \sIS (F)$ and $g= \sIS (G)$
in Lemma \ref{lem.piadditive}, it follows that
\[ \pi_* (\sIS (F+G))
   = \pi_* (\sIS (F) + \sIS (G))
   = \pi_* (\sIS (F)) + \pi_* (\sIS (G)). \]
\end{proof}

This applies for example to multiplication by an integer, which
is relevant for the up-down principle:

\begin{cor} \label{cor.multbydonih}
Let $d$ be an integer
and let $(Y,\Ya)$ be a connected (for intersection Dold-Thom) filtered polyhedron.
Let $h_d: IH^\bp_j (Y;\intg) \to IH^\bp_j (Y;\intg)$ and 
$a_d: AG(Y)\to AG(Y)$
denote the group endomorphisms given by multiplication by $d$.
Then, for any $\bp$, under the intersection Dold-Thom identification
$IH^\bp_j (Y;\intg) = I^\bp \pi_j (AG(Y)) = \pi_j (\sIS^\bp (AG(Y))),$
the placid map $a_d$ induces the endomorphism $h_d$.
\end{cor}

Here, we used the following basic observation:
\begin{lemma} \label{lem.multbydisplacidonagy}
The map $a_d: (AG(Y),AG(\Ya)) \longrightarrow (AG(Y), AG(\Ya))$
is placid.
\end{lemma}
\begin{proof}
If $d=0$, then $a_0 =0$ is placid by Lemma 
\ref{lem.sumofcofilterediscofandzeroisplacid}.
If $d\not= 0$, 
the key point is that $a_d$ does not change supports.
\end{proof}

\section{Ramified Covers and Intersection Homology Transfer}
\label{sec.ramcovsandihtransfer}

A finite ramified covering $\pi: \widetilde{Y} \to Y$ has an associated
transfer homomorphism
\[ \pi_!: H_* (Y;\intg) \to H_* (\widetilde{Y}; \intg) \]
such that
\begin{equation} \label{equ.transfermultdegree}
\pi_* \circ \pi_!: H_* (Y;\intg) \longrightarrow H_* (Y;\intg)
 \text{ is multiplication by } d,
\end{equation}
where $d$ is the degree of $\pi$ and
$\pi_*: H_* (\widetilde{Y};\intg) \longrightarrow H_* (Y; \intg)$
the covariantly induced homomorphism.
In this section, we will show, using the free abelian topological
groups generated by $Y, \wY$ in a simplicial setting,
that this transfer lifts to intersection
homology $IH^{\bar{p}}_*$. There is thus a transfer 
\[ \pi_!: IH^{\bar{p}}_* (Y;\intg) \longrightarrow 
    IH^{\bar{p}}_* (\widetilde{Y}; \intg) \]
for every $\bar{p}$, when $\pi$ is simplicial and
$Y, \wY$ PL stratified polyhedra. 
This transfer will be shown 
(see Proposition \ref{prop.updownmultbydegree})
to satisfy the above
property (\ref{equ.transfermultdegree}) on intersection homology.
Furthermore, it is compatible with the transfer on ordinary homology,
that is,
\begin{equation} \label{equ.ihtransfercompatibleclassicaltransf}
\xymatrix@R=20pt{
IH^{\bar{p}}_* (Y;\intg) \ar[d] \ar[r]^{\pi_!} & 
   IH^{\bar{p}}_* (\widetilde{Y};\intg) \ar[d] \\
H_* (Y;\intg) \ar[r]_{\pi_!} & H_* (\widetilde{Y};\intg)
} \end{equation}
commutes, where the vertical maps are the canonical ones.

We begin by recalling the notion of a finite ramified covering of
a topological space $Y$ in the elegant formulation of L. Smith, \cite{smithlarry}.
\begin{defn} \label{def.ramifiedcovering}
Let $d$ be a positive integer.
A \emph{$d$-fold ramified covering} is a pair
$(\pi, \mu)$, where
\begin{itemize}
\item $\pi: \widetilde{Y} \to Y$ is a continuous surjective finite-to-one map, and
\item $\mu: \widetilde{Y} \to \{ 1,2,3,\ldots \}$ is a map, called the
   \emph{multiplicity map}, 
\end{itemize}
such that
\begin{enumerate}
\item $\sum_{x \in \pi^{-1} (y)} \mu (x) = d$ for every $y \in Y$, and
\item the map
 \[  \tau_\pi: Y \longrightarrow SP^d (\widetilde{Y}) \]
  defined by sending $y$ into $\pi^{-1} (y),$
  where each $x \in \pi^{-1} (y)$ occurs $\mu (x)$ times, is continuous.
 \end{enumerate}
\end{defn}
We will often write ramified coverings briefly as $\pi: \wY \to Y$,
omitting the multiplicity map from the notation.

\begin{example} \label{exple.orbitprojisramifiedcov}
Let $G$ be a finite group and $X$ a $G$-space.
Then the orbit projection $\pi: X \to X/G$ admits
the structure of a $|G|$-fold ramified covering
(\cite[Prop. 1.5]{smithlarry}).
\end{example}

Ramified coverings can be pulled back under arbitrary continuous
maps: The total space is the topological fiber product and
the multiplicity map on the fiber product is taken to be $\mu$
of the projection to the component $\wY$
(\cite[p. 488, Prop. 1.3]{smithlarry}). The degree of the
pullback is again $d$. In particular,
ramified coverings can be restricted to subspaces of their target.
Let $\Sigma X$ denote the unreduced suspension of a space $X$.

\begin{lemma} \label{lem.suspensionisramcov}
Let $\pi: \wY \to Y$ be a ramified covering of degree $d$ 
with multiplicity map $\mu$.
Then the suspension $\Sigma \pi: \Sigma \wY \to \Sigma Y$ is a ramified
covering of degree $d$ with multiplicity map $\mu_\Sigma$ given by
$\mu_\Sigma (t,x) = \mu (x)$ for $(t,x) \in (-1,1) \times \wY$ and
$\mu_\Sigma (\pm 1) = d,$ where $\pm 1 \in \Sigma \wY$ are the two suspension points.
\end{lemma}
\begin{proof}
The map $\Sigma \pi$, which is given by $(\Sigma \pi)(t,x) = (t,\pi (x))$
and $(\Sigma \pi)(\pm 1) = \pm 1,$ is surjective and finite-to-one,
by the corresponding properties of $\pi$.
By definition of the multiplicity map $\mu_\Sigma$, the summation property (1)
in Definition \ref{def.ramifiedcovering} holds, noting that the 
fibers of $\Sigma \pi$ are given by
$(\Sigma \pi)^{-1} (t,y) = \{ t \} \times \pi^{-1} (y),$
 $(t,y) \in (-1,1) \times Y,$
and $(\Sigma \pi)^{-1} (+1) = \{ +1 \},$
$(\Sigma \pi)^{-1} (-1) = \{ -1 \}$.
It remains to verify property (2), i.e. that the map
$\tau_{\Sigma \pi}: \Sigma Y \to SP^d (\Sigma \wY)$
determined by $\Sigma \pi$ and $\mu_\Sigma$ is continuous.
Consider the cartesian diagram
\[ \xymatrix{ 
I \times \wY  \ar[r]^{\operatorname{proj}} 
  \ar[d]_{\overline{\pi} = \id \times \pi} & \wY \ar[d]^\pi \\
I \times Y \ar[r]^{\operatorname{proj}} & Y.
} \]
We recalled above that ramified covers can be pulled back.
Thus $\overline{\pi}$ is a ramified covering, whose multiplicity
function $\overline{\mu}: I \times \wY \to \{ 1,2,\ldots \}$ is given
by $\overline{\mu} (t,x) = \mu (x)$.
Therefore, we know the map
$\tau_{\overline{\pi}}: I\times Y \longrightarrow SP^d (I \times \wY),$
given by
\[ \tau_{\overline{\pi}} (t,y) = 
     \sum_{(t,x) \in \overline{\pi}^{-1} (t,y)} \overline{\mu} (t,x) \cdot (t,x)
      = \sum_{x\in \pi^{-1} (y)} \mu (x) \cdot (t,x),  \]
to be continuous.
Let
$q: I\times Y \longrightarrow \Sigma Y,$  
$\widetilde{q}: I\times \wY \longrightarrow \Sigma \wY$
denote the quotient maps.
The latter quotient map induces a continuous map
$SP^d (\widetilde{q}): 
    SP^d (I\times \wY) \longrightarrow SP^d (\Sigma \wY).$
One checks readily that the diagram
\[ \xymatrix{ 
I \times Y \ar[r]^{\tau_{\overline{\pi}}} \ar[d]_q & SP^d (I \times \wY) 
   \ar[d]^{SP^d (\widetilde{q})} \\
\Sigma Y \ar[r]_{\tau_{\Sigma \pi}} & SP^d (\Sigma \wY)
} \]
commutes.
Now by definition of the quotient topology on $\Sigma Y$,
the map $\tau_{\Sigma \pi}$ is continuous if and only if
$\tau_{\Sigma \pi} \circ q$ is continuous.
By the commutativity of the above diagram,
$\tau_{\Sigma \pi} \circ q = SP^d (\widetilde{q}) \circ \tau_{\overline{\pi}}$
and the latter is continuous as the composition of two continuous maps.
\end{proof}

For later use, we record the standard fact that simplicial maps are cellular
with respect to the skeletal filtrations. This is a consequence of the simplex-wise
linearity of such a map. 
\begin{lemma} \label{lem.simplicialmapiscellular}
(Cellularity of simplicial maps.)
Let $X,Y$ be triangulated polyhedra of the same dimension, equipped
with the simplicial filtrations $\Xa, \Ya$.
If $f:X\to Y$ is a simplicial map, then 
$f(X^k) \subset Y^k$
for all $k\geq 0$.
\end{lemma}

Let $\widetilde{Y}$ and $Y$ be spaces triangulated by countable
simplicial complexes and let $\pi: \widetilde{Y} \to Y$ be a
simplicial $d$-fold ramified cover.
Via the topological embedding
$SP^d (\widetilde{Y}) \hookrightarrow AG(\widetilde{Y}),$
we may regard the map $\tau_\pi$ as a continuous map
\[ \tau_\pi: Y \longrightarrow AG(\widetilde{Y}), \]
given by
$\tau_\pi (y) = \sum_{x\in \pi^{-1}(y)} \mu (x)\cdot x.$
Since $\mu (x)\geq 1$ for every $x\in \pi^{-1} (y)$,
the support is the fiber of $\pi$,
$\supp (\tau_\pi (y)) = \pi^{-1} (y).$
The map $\tau_\pi$ is separated in the sense of 
Definition \ref{def.separatedmapintoag},
since the fibers over different points are disjoint.

\begin{lemma} \label{lem.taupiequidimstratcofiltered}
Let $Y,\widetilde{Y}$ be finite-dimensional triangulated polyhedra.
Let $\pi: \widetilde{Y} \to Y$ be a simplicial ramified covering.
If $\Ya, \widetilde{\Ya}$ denote the filtrations of $Y,\widetilde{Y}$ given by
the simplicial skeleta, then the map 
\[ \tau_\pi: (Y,\Ya) \to 
  (AG(\widetilde{Y}), AG(\widetilde{\Ya})) \]
is equidimensionally stratified and cofiltered.
\end{lemma}
\begin{proof}
We begin by showing that $\tau_\pi$ is equidimensionally stratified.
Since $\pi$ is finite-to-one, surjective and simplicial,
the proof of Lemma \ref{lem.finsurjsimpmapiscofiltered}
applies to show that $\dim (Y) = \dim (\widetilde{Y}) =: n$.
Given $k\geq 0,$ we must show that
$\tau_\pi (Y_{n-k} - Y_{n-k-1}) \subset 
    AG(\widetilde{Y})_{n-k} - AG(\widetilde{Y})_{n-k-1}.$
Let 
\[ S=\Delta^\circ \subset Y_{n-k} - Y_{n-k-1} = Y^k - Y^{k+1} \]
be a stratum of $\Ya$, $\dim \Delta = n-k$.
Let $y\in S$ be a point.
Since $\pi$ is simplicial, the preimage of the open simplex has
the form
$\pi^{-1} (\Delta^\circ) \cong \Delta^\circ \times \pi^{-1} (y),$
see e.g. Milnor-Stasheff \cite[p. 236, Lemma 20.5]{milnorstasheff}.
The set $\pi^{-1} (y)$ is finite, as $\pi$ is finite-to-one.
Thus 
$\pi^{-1} (\Delta^\circ) = \Delta^{n-k,\circ}_1 \sqcup \cdots 
    \sqcup \Delta^{n-k,\circ}_j$
is a disjoint union of the interiors of $(n-k)$-dimensional simplices 
in $\widetilde{Y}$.
We must show that
$\tau_\pi (y) \in AG(\widetilde{Y})^k - AG(\widetilde{Y})^{k+1}.$
By definition of the Lawson filtration, this requires us to verify
$\supp (\tau_\pi (y)) \cap \widetilde{Y}^k \not= \varnothing$
and
$\supp (\tau_\pi (y)) \cap \widetilde{Y}^{k+1} = \varnothing.$
The support is given by the fiber
\begin{equation} \label{equ.supptaupiisfiber} 
\supp (\tau_\pi (y)) = \pi^{-1} (y), 
\end{equation}
which is nonempty, since $\pi$ is surjective.
This nonempty set is contained in
\[ \pi^{-1} (\Delta^\circ)
   = \Delta^{n-k,\circ}_1 \sqcup \cdots 
       \sqcup \Delta^{n-k,\circ}_j
       \subset \widetilde{Y}^k - \widetilde{Y}^{k+1}. \]
This shows that $\tau_\pi$ is equidimensionally stratified.

It remains to be shown that $\tau_\pi$ is cofiltered.
We need to verify that
$\tau^{-1}_\pi (AG(\widetilde{Y})^k) \subset Y^k$
for every $k\geq 0$.
Thus let $y\in Y$ be a point such that
$\tau_\pi (y) \in AG(\widetilde{Y})^k$.
Then
$\supp (\tau_\pi (y)) \cap \widetilde{Y}^k \not= \varnothing.$
Since the support is the $\pi$-fiber (\ref{equ.supptaupiisfiber}),
there exists a point $x\in \widetilde{Y}^k$ with $\pi (x)=y$.
Now since $\pi$ is a simplicial map between polyhedra of equal dimension,
Lemma \ref{lem.simplicialmapiscellular} implies that
$\pi (\widetilde{Y}^k) \subset Y^k.$
Therefore, $y=\pi (x) \in Y^k$, as was to be shown.
We conclude that $\tau_\pi$ is cofiltered.
\end{proof}

Let $Y$ and $\widetilde{Y}$ be finite-dimensional (countably) triangulated 
polyhedra. As usual, we endow these polyhedra with their simplicial
filtrations $\Ya, \widetilde{\Ya}$.
Let $\pi: \widetilde{Y} \to Y$ be a simplicial ramified covering.
By Lemma \ref{lem.taupiequidimstratcofiltered},
$\tau_\pi: (Y,\Ya) \to (AG(\widetilde{Y}), AG(\widetilde{\Ya}))$
is equidimensionally stratified and cofiltered.
By the universal property of Proposition \ref{prop.markovuniversal}
there exists a unique extension of $\tau_\pi$
to a continuous group homomorphism
\[ \tau: AG(Y) \longrightarrow AG (\widetilde{Y}). \]
Since $\tau_\pi$ is separated, equidimensionally stratified and
cofiltered, Lemma \ref{lem.extensionisequistrat} implies
that $\tau$ is equidimensionally stratified and cofiltered.
Explicitly, this map is given by
$\tau (\phi)(x) = \phi (\pi (x)) \mu (x).$
The function $\tau (\phi)$ is nonzero only at finitely many points of
$\widetilde{Y}$, since this is true for $\phi$, and $\pi$ is finite-to-one.
On the characteristic function $1_y: Y\to \intg$ of a point $y\in Y$,
$\tau$ is given by
\[
\tau (1_y)(x) =
\begin{cases}
\mu (x),& x \in \pi^{-1} (y), \\
0,& \text{ otherwise.}
\end{cases}
\]
Alternatively, we may describe $\tau$ in the form
\[ \sum_{y \in Y} n_y y \mapsto
   \sum_{y \in Y} n_y \sum_{x \in \pi^{-1} (y)} \mu (x) x.  \]
According to Lemma \ref{lem.stratifiedandcofilteredisplacid},
$\tau$ is placid. 
By Proposition \ref{prop.cofilteredinducesmaponsisp},
$\tau$ induces for every $\bar{p}$ a simplicial group homomorphism
\[  \sIS^\bp (\tau): \sIS^\bp (AG(Y),AG(\Ya)) \longrightarrow 
                  \sIS^\bp (AG(\wY),AG(\widetilde{\Ya})). \]
On homotopy groups, this simplicial map induces a homomorphism
\[  \tau_* := \pi_* (\sIS^\bp (\tau)): \pi_* \sIS^\bp (AG(Y),AG(\Ya)) 
     \longrightarrow 
       \pi_* \sIS^\bp (AG(\wY),AG(\widetilde{\Ya})). \]
In this way, we receive a homomorphism
\[ \tau_*: I^\bp \pi_* (AG(Y), AG (\Ya)) \longrightarrow
   I^\bp \pi_* (AG(\widetilde{Y}), AG (\widetilde{\Ya})) \]
between the intersection homotopy groups of $Y$ and $\wY$.
Assume that the polyhedra $Y,\widetilde{Y}$ are connected.
Then by Gajer's intersection Dold-Thom theorem, \cite{gajer} and \cite{gajercorrection},
there are natural isomorphisms
\[  IH^\bp_* (Y;\intg) \cong I^\bp \pi_* (AG(Y), AG (\Ya)),~
     IH^\bp_* (\widetilde{Y};\intg) \cong 
   I^\bp \pi_* (AG(\widetilde{Y}), AG (\widetilde{\Ya})). \]
Under these isomorphisms, $\tau_*$ is a homomorphism
\[ \pi_! := \tau_*: IH^\bp_* (Y;\intg) \longrightarrow  
   IH^\bp_* (\widetilde{Y};\intg). \]
This is the desired \emph{intersection homology transfer} of the
ramified covering $\pi: \wY \to Y$.
The commutativity of Diagram (\ref{equ.ihtransfercompatibleclassicaltransf}) 
follows from the commutativity of
\[
\xymatrix{
\sIS^\bp (AG(Y) \ar@{^{(}->}[d] \ar[r]^{\sIS^\bp (\tau)} & 
   \sIS^\bp (AG(\wY)) \ar@{^{(}->}[d] \\
\sS(AG(Y)) \ar[r]_{\sS (\tau)} & \sS (AG(\wY)).
} \]

If $Y$ has even dimension $2k$, then only the intersection homology
in the middle degree $k$ is relevant for the definition of $G$-signatures
in Section \ref{sec.gsignatures}.
In degrees up to the middle, and using the lower middle perversity,
the connectivity assumption on $Y$ and $\wY$ is not needed,
as we will explain next.
If the perversity function $\bp$ is taken to be the lower middle
perversity $\bm$, then intersection homology
\emph{up to the lower middle degree} is insensitive to suspension.
\begin{lemma} \label{lem.ihlowermiddleofsuspension}
(1) If $L^{2k}$ is a $2k$-dimensional pseudomanifold with 
(unreduced) suspension $\Sigma L^{2k}$, then
$IH^\bm_i (\Sigma L^{2k}) = IH^\bm_i (L^{2k})$
for all $i \leq k$.\\
(2) If $L^{2k+1}$ is a $(2k+1)$-dimensional pseudomanifold, then
$IH^\bm_i (\Sigma L^{2k+1}) = IH^\bm_i (L^{2k+1})$
for all $i\leq k$. 
\end{lemma}
This is due to the fact that in the lower dimensions specified by the
above bounds, allowable chains cannot touch the suspension points.
The lemma has the following consequence.
Suppose that $Y^n$ is a polyhedral pseudomanifold of dimension $n$.
Even if $Y^n$ is not connected, its suspension $\Sigma Y^n$ is.
Thus the intersection Dold-Thom theorem does apply to $\Sigma Y^n$,
yielding an isomorphism
$IH^\bp_* (\Sigma Y^n) \cong I^\bp \pi_* (AG(\Sigma Y^n))$ for every $\bp$.
By Lemma \ref{lem.ihlowermiddleofsuspension},
$IH^\bm_i (Y^n) \cong I^\bm \pi_i (AG(\Sigma Y^n))$
for all $i \leq \lfloor \frac{n}{2} \rfloor,$
even when $Y$ is not connected. The degree bound is the same for
$Y$ and $\wY$, since they have the same dimension $n$.
By Lemma \ref{lem.suspensionisramcov}, the suspension of $\pi$
is a ramified covering $\Sigma \pi: \Sigma \wY \to \Sigma Y$.
It has an associated continuous homomorphism
$\tau_{\Sigma}: AG(\Sigma Y) \to AG(\Sigma \wY)$.
Thus for $i \leq \lfloor \frac{n}{2} \rfloor,$ one obtains a transfer
\[  
  \pi_!: IH^\bm_i (Y;\intg) \cong I^\bm \pi_i (AG(\Sigma Y))
   \stackrel{\tau_{\Sigma *}}{\longrightarrow}
   I^\bm \pi_i (AG(\Sigma \widetilde{Y})) 
   \cong  IH^\bm_i (\widetilde{Y};\intg). 
\]
If $Y$ happens to be already connected, then this agrees with
the transfer as defined previously, since the diagram
\[ \xymatrix@R=13pt{
AG(\Sigma Y) \ar[r]^{\tau_\Sigma} & AG(\Sigma \wY) \\
AG(Y) \ar[u] \ar[r]_\tau & AG(\wY) \ar[u]
} \]
commutes, where the vertical homomorphisms are induced
by the inclusions $Y \hookrightarrow \Sigma Y,$ 
$\wY \hookrightarrow \Sigma \wY$ at $t=0$.

The following result affirms the stability of  
intersection homology up to the middle-degree
under an \emph{arbitrary} number of iterated suspensions.
Its proof is a straightforward induction on the number of suspensions,
using Lemma \ref{lem.ihlowermiddleofsuspension}.
\begin{lemma} \label{lem.ihmiddledegstablesuspension}
If $L^{2k}$ is a $2k$-dimensional pseudomanifold, then for every
$s=0,1,2,3,\ldots$ and for every $i\leq k$,
\[ IH^\bm_i (\Sigma^s L^{2k}) = IH^\bm_i (L^{2k}).  \]
\end{lemma}

By Lemma \ref{lem.finsurjsimpmapiscofiltered},
a finite surjective simplicial map $\pi$, in particular
a simplicial finite ramified covering $\pi: \widetilde{Y} \to Y$
(between simplicially stratified complexes), is placid and thus
induces a homomorphism
\[ \pi_*: IH^{\bar{p}}_* (\widetilde{Y}) \longrightarrow IH^{\bar{p}}_* (Y).  \]
In particular, by Example \ref{exple.orbitprojisramifiedcov},
a simplicial finite orbit projection $\pi: X \to X/G$
induces a map
$\pi_*: IH^{\bar{p}}_* (X) \to IH^{\bar{p}}_* (X/G).$
The map
$\pi_*: IH^{\bar{p}}_* (\widetilde{Y}) \to IH^{\bar{p}}_* (Y)$
can be described via the intersection Dold-Thom theorem as follows:
A continuous map $f: X \to Y$ between polyhedra induces a continuous homomorphism
$AG(f): AG(X) \to AG(Y)$.
If $f$ is cofiltered, then $AG(f)$ is cofiltered by Lemma \ref{lem.cofilteredpresbyag}.
Hence in this situation, 
by Lemma \ref{lem.cofilteredonaginducesoninthtpygrps}, $AG(f)$ induces a map
$f_*: I^\bp \pi_* (AG(X)) \to I^\bp \pi_* (AG(Y)).$
Now, as a simplicial ramified covering $\pi: \widetilde{Y} \to Y$ is
indeed cofiltered, $AG (\pi)$ induces 
$\pi_*: I^\bp \pi_* (AG(\widetilde{Y})) \longrightarrow
I^\bp \pi_* (AG(Y)).$

\begin{prop} \label{prop.updownmultbydegree}
(Up-Down.)
If $\pi: \wY \to Y$ is a $d$-fold simplicial ramified covering of connected polyhedra, then
\[   \pi_* \circ \pi_!: IH^\bp_* (Y;\intg) \longrightarrow IH^\bp_* (Y;\intg) \]
is multiplication by $d$.
Connectivity is not required when
$\bp = \bm$ and $* \leq \lfloor (\dim Y)/2 \rfloor$.
\end{prop}
\begin{proof}
Let $a_d: AG(Y) \to AG(Y)$ denote multiplication by $d$.
We claim that the diagram
\[ \xymatrix{
AG(Y) \ar[r]^\tau \ar[rd]_{a_d} & AG(\wY) \ar[d]^{AG(\pi)} \\
 & AG(Y)
} \]
commutes, i.e. that
$(AG(\pi) \circ \tau)(\phi) = d\cdot \phi.$
We proceed to verify this formula.
On an element $\phi: Y \to \intg$ of $AG(Y)$, we find at a point $y\in Y$,
\begin{align*}
((AG(\pi) \circ \tau)(\phi))(y)
&= (AG(\pi)(x \mapsto \phi (\pi (x)) \mu (x)))  (y) \\
&= \sum_{x \in \pi^{-1}(y)} \phi (\pi (x)) \mu (x) 
    = \sum_{x \in \pi^{-1}(y)} \phi (y) \mu (x) \\
&= \phi (y) \sum_{x \in \pi^{-1}(y)} \mu (x) 
   = \phi (y)\cdot d,
\end{align*}
as claimed.
Using Corollary \ref{cor.multbydonih},
\begin{align*}
\pi_* \circ \pi_!
&= \pi_* (\sIS^\bp (AG(\pi))) \circ \pi_* (\sIS^\bp (\tau)) 
   = \pi_* (\sIS^\bp (AG(\pi) \circ \tau)) \\
&= \pi_* (\sIS^\bp (a_d)) 
   = h_d.
\end{align*}
When one or both of the polyhedra $Y,\wY$ are not connected,
we apply the above argument to the suspension ramified cover 
$\Sigma \pi$ and obtain that $(\Sigma \pi)_* \circ (\Sigma \pi)_!$
is multiplication by $d$ on $IH^\bm_i (\Sigma Y)$. The result follows
by observing that $\pi_*$ corresponds to $(\Sigma \pi)_*$
under the identifications $IH^\bm_i (Y) = IH^\bm_i (\Sigma Y),$
$IH^\bm_i (\wY) = IH^\bm_i (\Sigma \wY),$
$i \leq \lfloor (\dim Y)/2 \rfloor,$ since these are induced by the
inclusions $Y \hookrightarrow \Sigma Y$,
$\wY \hookrightarrow \Sigma \wY,$ and these form a commutative
square with $\pi$ and $\Sigma \pi$.
\end{proof}
When rational (or real) coefficients are used,
the up-down formula of Proposition \ref{prop.updownmultbydegree}
implies that $\pi_!$ is a monomorphism and $\pi_*$ is an 
epimorphism.

The down-up property of the next proposition is an intersection homology
analog of \cite[Prop. 2.4, p.493]{smithlarry}. Basic material on simplicial group
actions is recalled in Section \ref{sec.simplicialactions}.
For example, subanalytic proper actions admit a $G$-equivariant triangulation,
\cite[Prop. 6.7]{blp}; see Examples \ref{exples.subanalyticsimplicial}.
\begin{prop} \label{prop.downupsumofgstar}
(Down-Up.)
Let $G$ be a finite group acting simplicially on the connected 
triangulated polyhedron $X$
with orbit map $\pi: X \to X/G = Y$ structured as a $d:= |G|$-fold
simplicial ramified covering. Then
\[   \pi_! \circ \pi_*: IH^\bp_* (X;\intg) \to IH^\bp_* (X;\intg) \]
is given by
\[ \pi_! \circ \pi_* = \sum_{g\in G} g_*, \]
where $g_*: IH^\bp_* (X) \to IH^\bp_* (X)$ is the automorphism induced
by $g\in G$.
Connectivity is not required when
$\bp = \bm$ and $* \leq \lfloor (\dim X)/2 \rfloor$.
\end{prop}
\begin{proof}
The multiplicity map $\mu: X \to \{ 1, 2,\ldots \}$ is given by
$\mu (x) = |G_x|,$ the order of the isotropy group at $x$.
Thus, the composition
\[ AG (X) \stackrel{AG(\pi)}{\longrightarrow} AG (X/G)  
  \stackrel{\tau}{\longrightarrow} AG (X)
\]
is given by (the linear extension of)
\begin{align*}
\tau ((AG(\pi)) (x))
&= \tau (\pi (x)) 
  = \sum_{x' \in \pi^{-1} (\pi (x))} \mu (x')~x' 
  = \sum_{x' \in \pi^{-1} (\pi (x))} |G_{x'}|~ x' \\
&= \sum_{x' \in G\cdot x} |\{ g\in G ~|~ gx=x' \}| x' 
  = \sum_{g\in G} gx.
\end{align*}
In other words:
For every $g\in G$, the (simplicial) homeomorphism $g: X\to X$ induces
an automorphism $AG(g): AG(X) \to AG(X)$ of topological abelian groups. 
By functoriality of $AG(-)$, this defines an action of $G$
on $AG(X)$. 
The simplicial homeomorphism $g:X \to X$ is equidimensionally stratified
and cofiltered with respect to the filtration of $X$ by the simplicial skeleta.
By Lemma \ref{lem.aggequistratcofiltered}, $AG(g)$ is
equidimensionally stratified and cofiltered with respect to the
Lawson filtration.
In particular, $AG(g)$ is placid for every $g\in G$ 
(Lemma \ref{lem.stratifiedandcofilteredisplacid}).
The sum
$\sum_{g\in G} AG(g) \in \End (AG(X))$ is an endomorphism on the abelian
topological group $AG(X)$ and we have
\[ \tau \circ AG(\pi) = \sum_{g\in G} AG(g). \] 
We note that the sum $\sum_g AG(g)$ is cofiltered by
Lemma \ref{lem.sumofcofilterediscofandzeroisplacid}.
One can also see this as follows:
The orbit projection $\pi: X \to X/G$ is
cofiltered (in fact even placid) with respect to the simplicial stratifications
by Lemma \ref{lem.finsurjsimpmapiscofiltered}.
According to Lemma \ref{lem.cofilteredpresbyag}, $AG (\pi)$ is cofiltered.
We have seen earlier that $\tau$ is placid, in particular cofiltered.
Since the composition of cofiltered maps is cofiltered
(Lemma \ref{lem.compcofilteredmapsiscofiltered}), the map
$\tau \circ AG(\pi)$ is cofiltered.
This shows that the sum $\sum_{g\in G} AG(g)$ is cofiltered.
Since $\pi_* \sIS^\bp$ is additive by
Proposition \ref{prop.piisisadditive}, and all involved maps are cofiltered,
the claim follows from the calculation
\begin{align*}
\pi_! \circ \pi_*
&= \pi_* (\sIS^\bp (\tau)) \circ \pi_* (\sIS^\bp (AG(\pi))) 
  = \pi_* (\sIS^\bp (\tau \circ AG(\pi))) \\
&= \pi_* (\sIS^\bp (\sum_{g\in G} AG(g))) 
  = \sum_{g\in G} \pi_* (\sIS^\bp (AG(g))) 
  = \sum_{g\in G} g_*.
\end{align*}

When $X$ is not connected, one uses the suspension $\Sigma X$
as follows:
The suspension is a $G$-space in the standard way, i.e.
$g\cdot (t,x) = (t, g\cdot x)$. The two suspension points are fixed points.
Its orbit space is homeomorphic to $\Sigma (X/G)$ via
$(t,x)^* \mapsto (t,x^*)$, where $(-)^*$ denotes the orbit of a point.
Under this homeomorphism, the orbit projection
$\hat{\pi}: \Sigma X \to (\Sigma X)/G$ identifies with
$\Sigma \pi: \Sigma X \to \Sigma (X/G)$.
Moreover, the homeomorphism induces an isomorphism on $AG(-)$,
under which $\hat{\tau}: AG ((\Sigma X)/G) \to AG(\Sigma X)$
corresponds to
$\tau_\Sigma: AG(\Sigma (X/G)) \to AG (\Sigma X)$.
Since $\Sigma X$ is connected, we already know that
$\hat{\pi}_! \circ \hat{\pi}_* = \sum_g g_*$ on $IH^\bm_i (\Sigma X)$.
The result follows since for $i\leq \lfloor (\dim X)/2 \rfloor$, 
$\pi_! \circ \pi_*$ corresponds under the isomorphisms
$IH^\bm_i (X) = IH^\bm_i (\Sigma X)$ and
$IH^\bm_i (X/G) = IH^\bm_i (\Sigma (X/G))$
(Lemma \ref{lem.ihlowermiddleofsuspension})
to $(\Sigma \pi)_! \circ (\Sigma \pi)_*$, which in turn corresponds to
$\hat{\pi}_! \circ \hat{\pi}_*$. Also note that the 0-level inclusion
$X \hookrightarrow \Sigma X$ is $G$-equivariant. Thus
$g_*$ on $IH^\bm_i (X)$ corresponds to $g_*$ on $IH^\bm_i (\Sigma X)$.
\end{proof}

\begin{prop} \label{prop.rattransferisisoontoinvariants}
Let $G$ be a finite group acting (simplicially) on the connected polyhedron $X$
with orbit map $\pi: X \to X/G = Y$ structured as a $d:= |G|$ fold
ramified covering. Then the transfer map $\pi_!$ with $\rat$-coefficients
restricts to an isomorphism
\[ \pi_!: IH^\bp_* (X/G;\rat) \stackrel{\simeq}{\longrightarrow} 
        IH^\bp_* (X;\rat)^G \]
onto the subspace $IH^\bp_* (X;\rat)^G \subset IH^\bp_* (X;\rat)$
of invariant elements.        
Connectivity is not required when
$\bp = \bm$ and $* \leq \lfloor (\dim X)/2 \rfloor$.
\end{prop}
\begin{proof}
We have already pointed out that $\pi_!$ is a monomorphism.
We show first that the image of $\pi_!$ is contained in $IH_* (X;\rat)^G$.
Given any element $v\in IH_* (X/G;\rat),$ we put
$w := \frac{1}{d} \pi_! (v) \in IH_* (X;\rat)$. 
(Here, rational coefficients are essential.)
Then, using the up-down Proposition \ref{prop.updownmultbydegree},
$\pi_* (w) = v$.
Using the down-up Proposition \ref{prop.downupsumofgstar} twice, 
an element $h\in G$ acts on $\pi_! (v)$ by
\begin{align*}
h_* (\pi_! (v))
& = h_* (\pi_! \pi_* (w)) 
  = h_* \left( \sum_{g\in G} g_* (w) \right) 
  = \sum_{g\in G} (hg)_* (w)  \\
&= \sum_{g\in G} g_* (w)  
  = \pi_! \pi_* (w) 
  = \pi_! (v).
\end{align*}
This shows that $\pi_! (v)$ is $G$-invariant.
Conversely, suppose that $w \in IH_* (X;\rat)^G$ is any invariant element.
Then $g_* (w)=w$ for all $g\in G$ and thus by the
down-up Proposition \ref{prop.downupsumofgstar},
\[ \pi_! \pi_* (w) = \sum_{g\in G} g_* (w) = |G|\cdot w = d\cdot w. \]
Consequently,
$w = \pi_! \left( \frac{1}{d} \pi_* (w) \right)$
is in the image of $\pi_!$, using rational coefficients.
The above argument applies also when $X$ is not connected,
$\bp = \bm$ and the degree $i$ satisfies $i \leq \lfloor (\dim X)/2 \rfloor$,
since the entire argument is homogeneously internal to that degree $i$.
\end{proof}

\begin{cor} \label{cor.pistarisorestrinvs}
For connected $X$,
the restriction of $\pi_*: IH^\bp_* (X;\rat) \to IH^\bp_* (X/G;\rat)$
to $IH^\bp_* (X;\rat)^G$ is an isomorphism
\[ \pi_*: IH^\bp_* (X;\rat)^G \stackrel{\simeq}{\longrightarrow} 
        IH^\bp_* (X/G;\rat). \]
Connectivity is not required when
$\bp = \bm$ and $* \leq \lfloor (\dim X)/2 \rfloor$.
\end{cor}
\begin{proof}
For improved clarity, we write $j: IH^\bp_* (X;\rat)^G \subset
IH^\bp_* (X;\rat)$ for the inclusion so that the restriction in question
is the composition $\pi_* \circ j$.
Let $w\in IH^\bp_* (X;\rat)^G$ be any element.
By Proposition \ref{prop.rattransferisisoontoinvariants},
there exists a (unique) $v\in IH^\bp_* (X/G;\rat)$ with 
$\pi_! (v)=w$. We will show that $\pi_* \circ j$ is injective.
Indeed, if $(\pi_* j)(w)=0$, then by the up-down
Proposition \ref{prop.updownmultbydegree},
\[ 0 = \pi_* j (w) = \pi_* j \pi_! (v) = d\cdot v,\]
which implies that $v=0$, and thus $w=0$.
For surjectivity, let $v$ be any element of $IH^\bp_* (X/G;\rat)$.
Then $w:= \frac{1}{d} \pi_! (v)$ is an element of $IH^\bp_* (X;\rat)$
with $\pi_* (w)=v$.
By Proposition \ref{prop.rattransferisisoontoinvariants}, $w$ is $G$-invariant.
\end{proof}
The above facts are also valid for real coefficients instead of rational
ones, indeed they hold over any field of characteristic zero.

\section{Duality Properties of Orbit Spaces}
\label{sec.dualitypropsoforbitspaces}

If a finite group acts algebraically on a quasi-projective complex algebraic
variety, then the orbit space is quasi-projective.
In particular, the orbit space can be stratified such that all strata
have even codimension and thus it satisfies the Witt condition.
We prove in this section that the above observation can be extended
from the algebraic setting to the purely topological setting, i.e.
the orbit space of the action of a finite group on a Witt space again
satisfies the Witt condition. In fact, we will establish a more general
statement: If the total space of a ramified covering is a Witt space,
then so is its base space (Theorem \ref{thm.baseoframifiedcoveriswitt}).
(We will deal with the propagation of the pseudomanifold condition
to the orbit space in Section \ref{sec.simplicialactions}.)
One may also draw a parallel to the
classical Conner conjecture, which asserts that the orbit space
of a finite group action on a $\rat$-acyclic space is again
$\rat$-acyclic. In fact, this conjecture holds more generally
for ramified coverings $\pi$ and is implied by the existence of 
an associated transfer homomorphism $\pi_!$ on ordinary homology
(L. Smith \cite[Cor. 2.4]{smithlarry}).
Closely related is the well-known fact that if a finite group acts on
a locally connected rational homology manifold, then the orbit space
is again a rational homology manifold (e.g. Bredon \cite{bredonsheaftheory}).
An important difference to these classical cases is that in order to establish
the Witt condition, we need a transfer on intersection homology.
Such a transfer is provided by Section \ref{sec.ramcovsandihtransfer}.

We begin with a lemma that can be deduced from 
M. Cohen's \cite[p. 194, Lemma 2.14]{cohenmarshallgenregnbhds}.
We provide a direct argument.
\begin{lemma} \label{sublemma1}
Let $f: X\to Y$ be a PL map between polyhedra and let $y\in Y$ be a point.
Then there exists a regular neighborhood $R(y)$ of $y$ in $Y$
such that
$f^{-1} R(y)$ is a regular neighborhood of the fiber $f^{-1} (y)$ in $X$.
\end{lemma}
\begin{proof}
Let $K,L$ be simplicial complexes with $|K|=Y,$ $|L|=X$ and
such that $f:L\to K$ is simplicial.
Let $K'$ be the subdivision of $K$ obtained by starring
at the point $y \in |K|$. Then $y$ is a vertex of $K'$.
There exists a simplicial subdivision $L'$ of $L$ such that
$f: L' \to K'$ is simplicial.
The single vertex subcomplex $\{ y \}$ is full in $K'$.
Therefore, if we define a simplicial map
$g: K' \to [0,1]$
by setting $g(y) =0$ and $g(v)=1$ for all vertices $v\in K'$, $v\not= y$,
then
$g^{-1} (0) = \{ y \}.$
Here, we regard the interval $[0,1]$ as a simplicial complex
with one $1$-simplex and the two endpoints as the only vertices.
Then
\[ R(y) := g^{-1} [0,\smlhf], \]
is a regular neighborhood of $y$ in $|K|$.
We shall verify that $f^{-1} R(y)$ is a regular neighborhood of
the fiber $F := f^{-1} (y)$, which is a subpolyhedron of $|L|$.
The fiber $F$ is triangulated by the subcomplex
$T := \{ \sigma \in L' : f(\sigma) = \{ y \} \}$ of $L'$.
This subcomplex is full in $L'$.
Define a simplicial map $g_F: L' \to [0,1]$ to be the composition $g \circ f$.
Its vanishing set is
$g_F^{-1} (0) = f^{-1} (g^{-1} (0)) = f^{-1} (y) =F.$
In particular, any vertex of $L'$ which is not in $T$ cannot be
mapped to $0$ by $g_F$ and must therefore be mapped to $1$, since
this is the only other vertex in $[0,1]$ available.
(This also shows that $T$ is full in $L'$.)
By regular neighborhood theory, the polyhedron
$g_F^{-1} [0,\smlhf]$ is a regular neighborhood of $F$, and
$g_F^{-1} [0,\smlhf] = f^{-1} (g^{-1} [0,\smlhf]) 
    = f^{-1} R(y).$
\end{proof}

\begin{lemma} \label{sublemma2}
Let $X$ be a polyhedron and $F = \{ x_1, \ldots, x_k \}$ 
a finite subset of $X$.
Then for any regular neighborhood $R(F)$ of $F$ in $X$,
there exists for every $i$ a regular neighborhood $R(x_i)$ of $x_i$ in $X$
such that there is a PL homeomorphism $X \to X$ which carries
$R(F)$ onto  $R(x_1) \sqcup \cdots \sqcup R(x_k)$ 
and which is the identity on $F$.
\end{lemma}
This is a consequence of the uniqueness theorem for 
regular neighborhoods of compact polyhedra in polyhedra
(see e.g. Rourke-Sanderson, p. 33, Thm. 3.8).
Lemmas \ref{sublemma1} and \ref{sublemma2} imply:

\begin{lemma} \label{sublemma3}
Let $f: X\to Y$ be a finite-to-one PL map and let $y \in Y$ be a point.
Then there exist a regular neighborhood $R(y)$ of $y$ in $Y$ 
and, for every point $x\in f^{-1} (y)$ in the fiber, 
a regular neighborhood $R(x)$ of $x$ in $X$
such that there is a PL homeomorphism
\[ f^{-1} R(y) \cong \bigsqcup_{x\in f^{-1} (y)} R(x) \]
which is the identity on $f^{-1} (y)$.
\end{lemma}

If $A\subset X$ is a subset of a topological space $X$, we
will write $\Bdy_X (A)$, or simply $\Bdy (A),$ for the 
topological boundary (frontier) of $A$ in $X$.
We will denote by $\Lk (p,P)$ the \emph{polyhedral} link of
a point $p$ in a polyhedron $P$. This link is a compact subpolyhedron
of $P$ that can be computed as follows: Choose a simplicial complex $K$
which triangulates $P$ and contains $p$ as a vertex. Then 
$\Lk (p,P)$ is PL homeomorphic to the polyhedron $|\Lk (p,K)|$,
where $\Lk (p,K)$ denotes the \emph{simplicial} link of the vertex $p$
in the complex $K$, see e.g. Armstrong \cite[p. 177]{armstrong}.
If $L$ is a subcomplex of a simplicial complex $K$, then
$N(L,K)$ denotes the simplicial neighborhood of $L$ in $K$.
The simplicial boundary of $N(L,K)$ is denoted by $\dot{N} (L,K)$.
\begin{lemma} \label{sublemma4}
Let $P$ be a polyhedron and $p\in P$ a point.
If $R(p)$ is any regular neighborhood of $p$, then there is a 
PL homeomorphism
$\Bdy R(p) \cong \Lk (p,P).$
\end{lemma}
\begin{proof}
We can write any given regular neighborhood as
$R(p) = |N(p,K')|,$ where
$K$ is a simplicial complex with $|K|=P$, $p$ is a vertex of $K$,
and $K'$ is a derived of $K$ near $p$.
By M. Cohen \cite[p. 194, Lemma 2.12]{cohenmarshallgenregnbhds},
$|\dot{N} (p,K')| = \Bdy |N(p,K')| = \Bdy R(p)$.
By their very definition, the simplicial complexes $\dot{N} (p,K')$
and $\Lk (p,K')$ are equal. The statement of the lemma follows,
as the polyhedron of the simplicial link $\Lk (p,K')$ is PL homeomorphic
to the polyhedral link $\Lk (p,P)$.
\end{proof}

\begin{lemma} \label{sublemma5}
Let $A \subset K$ be a full subcomplex.
Suppose that $g: K \to [0,1]$ is a simplicial map such that
$A = g^{-1} (0)$. Let $R(A)$ be the regular neighborhood of $|A|$
in $|K|$ given by $R(A) = g^{-1} [0,\smlhf]$. Then
$\Bdy R(A) = g^{-1} (\smlhf).$
\end{lemma}
\begin{proof}
This is a consequence of the fact that $\Bdy R(A)$ is bicollared
in $|K| - |A|$, see Cohen \cite[p. 203, Thm. 5.3]{cohenmarshallgenregnbhds}.
\end{proof}

\begin{lemma} \label{lem.preimageoflink}
Let $f: X\to Y$ be a finite-to-one PL map between polyhedra.
Then for every point $y\in Y$, there is 
a simplicial complex $K$ such that 
$|K|=Y,$ $y$ is a vertex of $K$, 
$|\Lk (y,K)| \cong \Lk (y,Y),$ and there is
a PL homeomorphism
\[ f^{-1} (|\Lk (y,K)|) \cong \bigsqcup_{x\in f^{-1} (y)} \Lk (x,X). \]
\end{lemma}
\begin{proof}
Let $R(y)$ be the regular neighborhood constructed in Lemma \ref{sublemma1}.
Thus $L', K'$ are simplicial complexes with $X=|L'|,$ $Y=|K'|$,
$y$ is a vertex of $K'$,
$f: L' \to K'$ is simplicial and $R(y) = g^{-1} [0,\smlhf],$
where $g: K' \to [0,1]$ is simplicial,
$g(y) =0$ and $g(v)=1$ for all vertices $v\in K'$, $v\not= y$;
we have $g^{-1} (0) = \{ y \}.$
Let $K''$ be a derived of $K'$ near $y$ such that
$R(y)=|N(y,K'')|$.
By Lemma \ref{sublemma5},
$\Bdy R(y) = g^{-1} (\smlhf).$ Hence
\[
f^{-1} (\Bdy R(y))
 = f^{-1} (g^{-1} (\smlhf))
 = g_F^{-1} (\smlhf),
\]
where $g_F: L' \to [0,1]$ is the simplicial map obtained by composing
$f$ and $g$, as in the proof of Lemma \ref{sublemma1};
$F = f^{-1} (y)$ denotes the fiber over $y$.
In that proof, 
$F$ was triangulated by a full subcomplex $T$ of $L'$, and
$g_F^{-1} (0) = F$, $g_F^{-1} [0,\smlhf] = f^{-1} R(y) =: R(F)$.
Using Lemma \ref{sublemma5} again,
$g_F^{-1} (\smlhf) = \Bdy R(F).$
This shows that
$f^{-1} (\Bdy R(y)) = \Bdy R(F).$
(We note on the side that
given a continuous map, the operation of taking preimages
does not in general commute with the operation of taking
the topological boundary of a subset.)
By Lemma \ref{sublemma3}, using that $f$ is finite-to-one,
every point $x\in f^{-1} (y)$ in the fiber possesses a regular neighborhood
$R(x)$ such that there is a PL homeomorphism
$R(F) = f^{-1} R(y) \cong \bigsqcup_{x\in f^{-1} (y)} R(x).$
By disjointness of this finite union of closed subspaces,
$\Bdy R(F) \cong \bigsqcup_{x\in f^{-1} (y)} \Bdy R(x).$
So
\[ f^{-1} (\Bdy R(y)) \cong \bigsqcup_{x\in f^{-1} (y)} \Bdy R(x). \]
As explained in the proof of Lemma \ref{sublemma4}, 
\[ \Bdy R(y)
 = \Bdy |N(y,K'')| 
 =  |\dot{N} (y,K'')| 
 = |\Lk (y,K'')|
  \cong \Lk (y,Y). \]
According to Lemma \ref{sublemma4}, there are PL homeomorphisms
$\Bdy R(x) \cong \Lk (x,X)$
for every $x\in F.$ 
The desired statement follows.
\end{proof}

Let $L,K$ be simplicial complexes and
let $f: L\to K$ be a simplicial map. Let $\Delta$ be a 
simplex of $K$ and let $y\in \Delta^\circ$ be a point in the
interior of $\Delta$.
We define a subset of the complex $L$ by
\[ f^* (\Delta) := \{ \Delta' \in L ~|~ 
   f^{-1} (y) \cap \Delta'^\circ \not= \varnothing \}.  \]
\begin{lemma} \label{lem.fstardelta}
If $f$ is finite-to-one, then
the preimage of the open simplex $\Delta^\circ$ is given by
\[ f^{-1} (\Delta^\circ) = 
      \bigsqcup_{\Delta' \in f^* (\Delta)} \Delta'^\circ,  \]
and
$\dim \Delta' = \dim \Delta$ for all $\Delta' \in f^* (\Delta).$
In particular, $f^* (\Delta)$ is independent of the choice
of $y \in \Delta^\circ$.
\end{lemma}
\begin{proof}
We show first that if $\Delta' \in f^* (\Delta),$
then $\Delta'^\circ \subset f^{-1} (\Delta^\circ)$.
There exists a point $x\in \Delta'^\circ$ such that $f(x)=y$.
Let $v_0,\ldots, v_r$ be the vertices of $\Delta'$, $r=\dim \Delta'$.
Then
$x = \sum_{i=0}^r t_i v_i$
for positive real numbers $t_i$ with sum $1$.
Thus
\begin{equation} \label{equ.yposconvcombinintdelta} 
y = f(x) = t_0 f(v_0) + \cdots + t_r f(v_r) \in \Delta^\circ. 
\end{equation}
Since $f$ is simplicial, the $f(v_i)$ are vertices of $K$
and the set $f(v_0), \ldots, f(v_r)$ spans a simplex $\sigma$
of $K$. 
Note that since $f$ is finite-to-one, $f(v_i) \not= f(v_j)$
whenever $i\not= j$. Therefore, $\dim \sigma =r$.
Since all $t_i$ are positive, (\ref{equ.yposconvcombinintdelta}) 
shows that $y$ is in the 
interior of $\sigma$. But it is also in the interior of $\Delta$.
Now, a point of a triangulated polyhedron is interior to precisely one simplex
of its triangulation. Hence $\sigma = \Delta$. 
This shows that every $f(v_i)$ is a vertex of $\Delta$, and
that $\dim \Delta = \dim \sigma = r = \dim \Delta'$.
Given any point $x' \in \Delta'^\circ$, we can write it as a convex
combination
$x' = \sum_{i=0}^r s_i v_i$
for positive real numbers $s_i$ with sum $1$.
Then
$f(x') = \sum s_i f(v_i) \in \Delta^\circ.$
This shows that $\Delta'^\circ \subset f^{-1} (\Delta^\circ)$.

Conversely, we prove that $f^{-1} (\Delta^\circ) \subset \bigsqcup \Delta'^\circ$:
Let $x'\in |L|$ be any point with $f(x') \in \Delta^\circ$.
Let $\Delta'$ be the unique simplex of $L$ such that
$x' \in \Delta'^\circ$.
We claim that $\Delta' \in f^* (\Delta)$.
Thus we need to show that $\Delta'^\circ$ contains a point of $f^{-1}(y)$.
Let $v_0,\ldots, v_r$ be the vertices of $\Delta'$.
Then $x' = \sum t_i v_i$
for positive real numbers $t_i$ with sum $1$.
Thus
$f(x') = \sum t_i f(v_i) \in \Delta^\circ.$
This shows, as above, that every $f(v_i)$ is a vertex of $\Delta$ and that,
since $f(x')$ is an interior point of $\Delta$, the set
$f(v_0),\ldots, f(v_r)$ constitutes the \emph{complete} list of
vertices of $\Delta$, i.e. every vertex of $\Delta$ is of the
form $f(v_i)$ for some $i$.
Hence, $y \in \Delta^\circ$ can be written as a convex combination
$y = \sum s_i f(v_i) \in \Delta^\circ$
with positive $s_i$.
Set 
$x := \sum s_i v_i.$
Then $x\in \Delta'^\circ$ is such that $f(x) =y$.
This shows that $\Delta' \in f^* (\Delta)$.
\end{proof}

Given a simplicial complex $K$ (always assumed to be locally finite) 
and a simplex $\Delta$ in $K$,
we denote the simplicial link of $\Delta$ in $K$ by $\Lk (\Delta, K)$.
This is a (finite) simplicial complex.
\begin{lemma} \label{lem.restrramcovonlinks}
Let $Y,\wY$ be pseudomanifolds triangulated by simplicial complexes
$L,K;$ $\wY = |L|,$ $Y= |K|$.
Let $\pi: \wY \to Y$ be a simplicial ramified covering with respect to $L,K$.
Let $\Delta$ be any simplex in $K$ and let $s=\dim \Delta$.
Then there is a PL map 
\[ \lambda: \bigsqcup_{\Delta' \in \pi^* (\Delta)} \Sigma^s |\Lk (\Delta',L)|
  \longrightarrow \Sigma^s |\Lk (\Delta,K)|    \]
which is a ramified covering.
\end{lemma}
\begin{proof}
Let $y\in \Delta^\circ$ be any point in the interior of $\Delta$.
Then by Akin \cite[p. 420, c $\to$ a]{akin}, 
there is a PL homeomorphism
\begin{equation} \label{equ.plkyyhomeosigmaslkdeltak}
\Lk (y,Y) \cong \Sigma^s |\Lk (\Delta, K)|, 
\end{equation}
see also  Armstrong \cite[p. 178]{armstrong}.
We noted earlier that the restriction of a $d$-fold ramified cover over
any subspace of its target is again a $d$-fold ramified covering.
Hence the restriction of $\pi$ defines a ramified covering
$\pi|: \pi^{-1} (|\Lk (y,K'')|) \to |\Lk (y,K'')|,$
where $K'$ is a simplicial subdivision of $K$ such that $y$ is a vertex 
of $K'$ and $K''$ is a derived of $K'$ near $y$.
Composition with the PL homeomorphism
$\pi^{-1} (|\Lk (y,K'')|) \cong \bigsqcup_{x\in \pi^{-1} (y)} \Lk (x,\wY)$
provided by Lemma \ref{lem.preimageoflink},
yields a PL map
$\lambda_1: \bigsqcup_{x\in \pi^{-1} (y)} \Lk (x,\wY) \to  
     |\Lk (y,K'')|.$
Under the homeomorphism to the disjoint union, the ramification data,
i.e. the multiplicity function, can be transferred from
$\pi^{-1} (|\Lk (y,K'')|)$ to $\bigsqcup_{x\in \pi^{-1} (y)} \Lk (x,\wY)$.
Then $\lambda_1$ is again a ramified covering.
The pullback of $\lambda_1$ under a PL homeomorphism
$\Lk (y,Y) \cong |\Lk (y,K'')|$ is then a ramified covering
\[ \lambda_0: \bigsqcup_{x\in \pi^{-1} (y)} \Lk (x,\wY) \longrightarrow  
     \Lk (y,Y).  \]
Given $x\in \wY,$ let $\Delta'_x$ be the unique simplex
of $L$ such that $x\in \Delta'^\circ_x$.
Then 
$\pi^* (\Delta) = \{ \Delta'_x ~|~ x \in \pi^{-1} (y) \}.$
By Lemma \ref{lem.fstardelta},
$\pi^{-1} (\Delta^\circ) = 
      \bigsqcup_{x\in \pi^{-1} (y)} \Delta'^\circ_x,$
and
$\dim \Delta'_x = \dim \Delta =s$ for all $x\in \pi^{-1} (y).$
Again by Akin \cite[p. 420, c $\to$ a]{akin},
\begin{equation} \label{equ.plkxwyhomeosigmaslkdeltaxl} 
\Lk (x,\wY) \cong \Sigma^s |\Lk (\Delta'_x, L)| 
\end{equation}
for every $x\in f^{-1} (y)$, since $x$ is interior to $\Delta'_x$.
Under the PL homeomorphisms
(\ref{equ.plkyyhomeosigmaslkdeltak}) and
(\ref{equ.plkxwyhomeosigmaslkdeltaxl}),
the ramified cover $\lambda_0$ can be written as a ramified cover
\[ \lambda: \bigsqcup_{x\in \pi^{-1} (y)} \Sigma^s |\Lk (\Delta'_x, L)| 
  \longrightarrow \Sigma^s |\Lk (\Delta, K)|.  \]
\end{proof}

\begin{thm} \label{thm.baseoframifiedcoveriswitt}
Let $\wY$ and $Y$ be triangulated oriented pseudomanifolds.
Let $\pi: \wY \to Y$ be a simplicial ramified covering (of finite degree).
If $\wY$ satisfies the Witt condition, then so does $Y$.
\end{thm}
\begin{proof}
If the Witt condition holds in some stratification, then it
holds in every stratification by the Proposition in
\cite[Section 2.4]{gmih2}.
We equip $\wY$ and $Y$ with the simplicial
stratifications $\widetilde{\Ya}$ and $\Ya$. The filtration
subspaces $\wY_i$ and $Y_i$ are thus given by the $i$-dimensional
simplicial skeleta $L_i, K_i$ of the triangulations
$|L| = \wY,$ $|K| = Y$.
Note that the link of any simplex in a simplicial pseudomanifold
is a (compact) pseudomanifold,
see e.g. Siegel \cite[p. 1070, I.2]{siegel}.

Let $\Delta \in K$ be a simplex whose simplicial link $\Lk (\Delta, K)$
has even dimension $2k$. 
We have to prove that $IH^\bm_k (|\Lk (\Delta, K)|; \rat) =0$.
By Lemma \ref{lem.restrramcovonlinks}, the restriction of
$\pi$ induces a ramified covering
\[ \lambda: \bigsqcup_{\Delta' \in \pi^* (\Delta)} \Sigma^s |\Lk (\Delta',L)|
  \longrightarrow \Sigma^s |\Lk (\Delta,K)|,    \]
where $s=\dim \Delta$.  
Let $B$ denote the target of $\lambda$ and $\wB$ its source.
By Lemma \ref{lem.fstardelta},
$\dim \Delta' = \dim \Delta$ for every $\Delta' \in \pi^* (\Delta)$.
Since $\pi$ is finite-to-one and surjective, $\dim \wY = \dim Y$
by (the proof of) Lemma \ref{lem.finsurjsimpmapiscofiltered}.
Therefore, $\dim Lk(\Delta',L) = \dim Lk(\Delta,K) = 2k$ and
$B$ and $\wB$ are PL pseudomanifolds of dimension
$\dim B = s + 2k = \dim \wB.$
Neither $\wB$ nor $B$ (when $s=0$) needs to be connected.
Since our description of the transfer uses Gajer's
intersection homology Dold-Thom theorem, which seems to be
available only for connected polyhedra,
it is perhaps not immediately clear whether $\lambda$ possesses a transfer
homomorphism on intersection homology.
In any case, if necessary, this technical issue can be overcome 
by the following method:
The suspension of $\lambda$ is a map
$\Sigma \lambda: \Sigma \wB \to \Sigma B$
between \emph{connected} PL pseudomanifolds, which is again
a PL ramified cover by Lemma \ref{lem.suspensionisramcov}.
Therefore, it has an associated transfer
$(\Sigma \lambda)_!: IH^\bm_* (\Sigma B) \to IH^\bm_* (\Sigma \wB),$
whose general construction was provided in Section \ref{sec.ramcovsandihtransfer}.
Using Lemma \ref{lem.finsurjsimpmapiscofiltered} with respect to 
appropriate simplicial subdivisions,
the placid map $\Sigma \lambda$ induces covariantly a homomorphism
$(\Sigma \lambda)_*: IH^\bm_* (\Sigma \wB)
  \to IH^\bm_* (\Sigma B).$
By Proposition \ref{prop.updownmultbydegree}, the composition
\[ \xymatrix{
IH^\bm_k (\Sigma B; \rat) \ar[r]^>>>>>{(\Sigma \lambda)_!} \ar[rd]_d
    & IH^\bm_k (\Sigma \wB; \rat) 
       \ar[d]^{( \Sigma \lambda)_*} \\
  & IH^\bm_k (\Sigma B; \rat)    
} \]
is multiplication by the degree $d\geq 1$ of the ramified covering $\Sigma \lambda$
(which equals the degree of $\lambda$).
(Using the coefficient isomorphism
$IH^\bm_* (X;\intg) \otimes \rat = IH^\bm_* (X;\rat)$, we can pass
from integral information to rational one.)
By Lemma \ref{lem.ihlowermiddleofsuspension},
$IH^\bm_k (\Sigma B) = IH^\bm_k (B),$
 $IH^\bm_k (\Sigma \wB) = IH^\bm_k (\wB).$
So the above commutative diagram can be desuspended
 to $\wB$ and $B$ and we obtain a commutative diagram
\[ \xymatrix@R=30pt{
IH^\bm_k (\Sigma^s |\Lk (\Delta,K)|; \rat) 
  \ar[r]^>>>>>{(\Sigma \lambda)_!} \ar[rd]_d
    & \bigoplus_{\Delta' \in \pi^* (\Delta)} IH^\bm_k (\Sigma^s |\Lk (\Delta',L)|; \rat) 
       \ar[d]^{(\Sigma \lambda)_*} \\
  & IH^\bm_k (\Sigma^s |\Lk (\Delta,K)|; \rat).    
} \]
Now, on a rational vector space, multiplication by $d\not= 0$ is an isomorphism
with inverse given by multiplication by $1/d$.
Since $\dim \Lk (\Delta',L) = 2k,$
Lemma \ref{lem.ihmiddledegstablesuspension} implies that
\[ IH^\bm_k (\Sigma^s |\Lk (\Delta',L)|; \rat)
   = IH^\bm_k (|\Lk (\Delta',L)|; \rat). \]
Since $\wY$ satisfies the Witt condition by assumption, we have
$IH^\bm_k (|\Lk (\Delta',L)|; \rat) =0$ for every $\Delta' \in \pi^* (\Delta)$.
This shows that the isomorphism $(\Sigma \lambda)_* \circ (\Sigma \lambda)_!$
factors through the zero vector space and is thus zero.
This implies that $IH^\bm_k (\Sigma^s |\Lk (\Delta,K)|; \rat) = 0$.
By Lemma \ref{lem.ihmiddledegstablesuspension},
\[ IH^\bm_k (|\Lk (\Delta,K)|; \rat) = 
   IH^\bm_k (\Sigma^s |\Lk (\Delta,K)|; \rat) = 0, \]
as was to be shown.\\
\end{proof}

\section{Simplicial Actions and Pseudomanifolds}
\label{sec.simplicialactions}

We recall here some basic material on simplicial group actions.
The main result of this section is Theorem \ref{thm.orbitspacepseudomanifold},
which states that when a finite group acts simplicially and orientedly
on an oriented pseudomanifold, then the orbit space is again an oriented
pseudomanifold.

Let $K$ be an (abstract) simplicial complex with
associated polyhedron $X=|K|$.
A discrete group $G$ acts on $X$ \emph{simplicially} 
(with respect to $K$) if every
transformation $g: X\to X$ is given by a simplicial map
$g:K\to K$, $g\in G$.
A simplicial action induces a simplicial action on the barycentric subdivision
of $K$.
The action on $K$ is called \emph{regular}, if for each subgroup 
$H$ of $G$ the following condition holds:
If $g_0, \ldots, g_n$ are elements of $H$ and
$(v_0,\ldots, v_n)$ and $(g_0 v_0,\ldots, g_n v_n)$
are both simplices of $K$, then there exists an element
$g\in H$ such that $g(v_i)=g_i (v_i)$ for all $i$.
(Here the vertices $v_0, \ldots, v_n$ need not be distinct.)
This condition implies that if $v$ and $gv$ belong to the
same simplex, then $gv=v$.
Thus if $g(\sigma)= \sigma$ for some simplex $\sigma$,
then $g$ acts as the identity on every vertex of $\sigma$,
and thus on every point of $\sigma$.
Any simplicial action can be turned into a regular one by
passing to the second barycentric subdivision.
Thus, without loss of generality, we can and will assume from now
on that every simplicial action is regular.
See Bredon \cite[p. 114ff]{bredontransfgroups} 
and Illman \cite{illman} for foundational material
on simplicial and regular actions.
 
Then a simplicial complex $K/G$ is given by taking the vertices
of $K/G$ to be the $G$-orbits $v^* := Gv$ of the vertices $v$ of $K$,
and by taking the simplices of $K/G$ to be tuples of the form
$(v^*_0, \ldots, v^*_n)$, where $(v_0,\ldots, v_n)$ is a simplex of $K$.
(This does \emph{not} mean that given a simplex
$(v^*_0, \ldots, v^*_n)$ of $K/G$, then $(v_0,\ldots, v_n)$
is a simplex of $K$. In this situation, one knows only that for each $i$,
there exists a vertex $w_i \in v^*_i$, such that 
$(w_0,\ldots, w_n)$ is a simplex of $K$.)
By regularity, if two simplices of $K$ lie over the same simplex of
$K/G$, then they can be moved to each other by a single element of $G$.
That is, if $(v_0,\ldots, v_n)$ and $(w_0, \ldots, w_n)$ are in $K$ such that
$(v^*_0,\ldots, v^*_n) = (w^*_0, \ldots, w^*_n)$,
then there exists a $g\in G$ with
$g (v_0,\ldots, v_n) = (w_0, \ldots, w_n)$.
This notation means that there is a permutation $s$ of $\{ 0,\ldots, n \}$
such that $g v_i = w_{s(i)}$.
The assignment $v \mapsto v^*$ defines a simplicial map $K\to K/G$,
whose associated continuous map
$X=|K| \to |K/G|$ can be identified canonically with the orbit projection
$X \to X/G = |K|/G$.
This shows that the orbit projection $\pi: X\to X/G$ is a simplicial
map with respect to the above triangulations. We will also write $K^*$
for the complex $K/G$. \\

\begin{examples} \label{exples.subanalyticsimplicial}
In \cite[Prop. 6.7]{blp}, Leichtnam, Piazza and the author proved that
subanalytic proper actions admit
a $G$-equivariant triangulation.
The proof uses the subanalytic triangulation theorem 
\cite[Cor. 3.5]{matumotoshiota} of Matumoto and Shiota, 
and Illman's general equivariant triangulation theorem
\cite[p. 497, Thm. 5.5]{illmangenequivtriang}.
Recall that a topological group $G$ is called \emph{subanalytic}
if it is contained in some real analytic manifold $M$ as a subanalytic
subset.
Every finite group is a subanalytic group. 
We assume that if a subanalytic group $G \subset M$ acts
on a subanalytic set $X \subset N$, then it does so subanalytically,
i.e. the graph of the action $G \times X \to X$ is subanalytic in
$M\times N \times N$.
Now let $X$ be a locally compact subanalytic set and let $G$ be a
subanalytic proper transformation group of $X$.
Then, according to \cite[Prop. 6.7]{blp},
$X$ admits a $G$-equivariant triangulation.
\end{examples}

The pseudomanifold property is preserved by orientation preserving
(regular) simplicial actions:
\begin{thm} \label{thm.orbitspacepseudomanifold}
Let $K$ be a simplicial complex which is an oriented pseudomanifold.
Suppose that $K$ is equipped with a (regular) simplicial
action of a finite group $G$ which preserves the orientation. 
Then the orbit complex $K/G$ is a pseudomanifold, and it is oriented.
\end{thm}
\begin{proof}
We recall that a simplicial complex $K$ is, by definition, an 
\emph{$n$-dimensional pseudomanifold} (without boundary),
if 
\begin{enumerate}
\item[(PM1)] every simplex of $K$ is a face of some $n$-simplex of $K$, and
\item[(PM2)] every $(n-1)$-simplex of $K$ is a face of precisely two $n$-simplices
 of $K$.
\end{enumerate}
(See e.g. \cite[p. 137]{gmih1}, \cite[p. 74, Examples 4.1.3]{banagltiss}.)
We must verify (PM1) and (PM2) for the complex $K^* = K/G$.

We establish first the following auxiliary claim:
If $(v_0, \ldots, v_n)$ is an $n$-dimensional simplex of $K$,
then its image simplex $(v^*_0, \ldots, v^*_n)$ in $K^*$
also has dimension $n$.
Indeed, suppose that the image simplex had dimension strictly less than $n$.
Then $v^*_i = v^*_j$ for some indices $i\not= j$.
Thus there exists $g\in G$ such that $gv_i = v_j$.
Consequently, $v_i$ and $gv_i$ are vertices of the same simplex of $K$.
Thus the regularity of the action implies that $gv_i = v_i$.
This yields the contradiction $v_i = v_j$. Therefore,
$(v^*_0, \ldots, v^*_n)$ has dimension $n$, as claimed.

Let us now prove property (PM1) for $K^*$.
Let $(v^*_0, \ldots, v^*_d)$ be any simplex in $K^*$,
where $(v_0,\ldots, v_d)$ is a simplex in $K$
with $v_i \not= v_j$ for all $i\not= j$.
Since $K$ is a pseudomanifold, there is an $n$-simplex
of the form 
$(v_0,\ldots, v_d, v_{d+1}, \ldots, v_n)$ in $K$.
Then the simplex $(v^*_0, \ldots, v^*_d)$ is the face of the simplex 
$(v^*_0,\ldots, v^*_d, v^*_{d+1}, \ldots, v^*_n)$ in $K^*$.
Note that by the auxiliary claim, the simplex
$(v^*_0,\ldots, v^*_n)$ is indeed $n$-dimensional.
This verifies property (PM1) for $K^*$.

We turn to (PM2) for $K^*$.
Let $(v^*_0, \ldots, v^*_{n-1})$ be any $(n-1)$-simplex in $K^*$,
where the tuple $(v_0, \ldots, v_{n-1})$ is an $(n-1)$-simplex in $K$.
Since $K$ is a pseudomanifold, there exist vertices $v_+,v_-,$ 
$v_+ \not= v_-,$ in $K$
such that
$(v_0, \ldots, v_{n-1}, v_+)$ and $(v_0, \ldots, v_{n-1}, v_-)$
are $n$-simplices in $K$.
Furthermore, there is no vertex $v\in K$ such that
$(v_0, \ldots, v_{n-1}, v)$ is
an $n$-simplex in $K$ and $v\not\in \{ v_+, v_- \}$.
Then $(v^*_0, \ldots, v^*_{n-1})$ is a face of both
$(v^*_0, \ldots, v^*_{n-1},v^*_+)$ and $(v^*_0, \ldots, v^*_{n-1},v^*_-)$.
These are simplices of $K^*$, and
we will show that they are two \emph{different} simplices. 
(Each one is indeed $n$-dimensional by the auxiliary claim.)
Suppose, by contradiction, that $v^*_+ = v^*_-$.
Then we have two simplices
$(v_0, \ldots, v_{n-1}, v_+)$ and $(v_0, \ldots, v_{n-1}, v_-)$
in $K$ that lie over the same simplex
$(v^*_0, \ldots, v^*_{n-1}, v^*_+) = (v^*_0, \ldots, v^*_{n-1}, v^*_-)$
of $K^*$. Thus the regularity of the action implies that there is
an element $g\in G$ such that
$g (v_0, \ldots, v_{n-1}, v_+) = (v_0, \ldots, v_{n-1}, v_-).$
Hence there exists a permutation $s$ of $\{ 0,\ldots, n-1,- \}$
such that
$g v_j = v_{s(j)},~ \forall j \in \{ 0,\ldots, n-1 \},$
and
$g v_+ = v_{s(-)}.$
For any $j \in \{ 0,\ldots, n-1 \},$
$v_j$ and $v_{s(j)}$ are vertices of the simplex
$(v_0, \ldots, v_{n-1}, v_-)$ in $K$. Therefore, since $v_{s(j)} = gv_j$,
the vertices $v_j$ and $gv_j$ belong to the same simplex of $K$. 
Thus by regularity of the action, $gv_j = v_j$ for all $0\leq j \leq n-1$.
It follows that $gv_+ = v_-$. (The permutation $s$ is thus the identity.)
So $g$ acts locally near points in the interior of the simplex
$(v_0, \ldots, v_{n-1})$ as a reflection across the hyperplane
spanned by $v_0, \ldots, v_{n-1}$.
Let $\Sigma \subset |K|=X$ denote the polyhedron of the 
$(n-2)$-skeleton of $K$. Since $K$ is an oriented pseudomanifold, 
$X-\Sigma$ is an oriented manifold (without boundary)
which is open and dense in $X$.
Let $A \subset X$ be the union of the simplices
$(v_0, \ldots, v_{n-1}, v_+)$ and $(v_0, \ldots, v_{n-1}, v_-)$.
Then the interior $U$ of $A$ is a $g$-invariant open neighborhood
of the interior of $(v_0, \ldots, v_{n-1})$, and $U$ is contained in
the oriented manifold $X-\Sigma$.
The orientation of $X-\Sigma$ induces an orientation of $U$.
The transformation $g$ acts on $U$ as the reflection across
a hyperplane and such a reflection reverses the orientation.
This contradicts the assumption that $G$ acts orientation preservingly.
It follows that $v^*_+ \not= v^*_-.$
This shows that $(v^*_0, \ldots, v^*_{n-1})$ is the face of at least
two different $n$-simplices in $K^*$. 
It remains to show that it is not the face of any other, third, $n$-simplex.
Thus, suppose that
$(v^*_0, \ldots, v^*_{n-1}, w^*)$ is an  $n$-simplex of $K^*$;
we must show that $w^* \in \{ v^*_+, v^*_- \}$.
Then there exists an  $n$-simplex $(w_0, \ldots, w_n)$ in $K$ with
$w^*_j = v^*_j$ for all $j=0,\ldots, n-1,$ and
$w^*_n = w^*$.
Thus $(w_0, \ldots, w_{n-1})$ and $(v_0, \ldots, v_{n-1})$
are two $(n-1)$-simplices in $K$ that both lie over the simplex
$(w^*_0, \ldots, w^*_{n-1}) = (v^*_0, \ldots, v^*_{n-1})$ in $K^*$.
The regularity of the action implies that there exists a $g\in G$
such that
$g (w_0, \ldots, w_{n-1}) = (v_0, \ldots, v_{n-1})$.
Consider the translate $g (w_0, \ldots, w_{n-1}, w_n)$.
This is a simplex in $K$, since $g$ acts simplicially.
It is $n$-dimensional, since $(w_0, \ldots, w_n)$ is; and
we have  $g (w_0, \ldots, w_{n-1}, w_n) =  (v_0, \ldots, v_{n-1}, gw_n)$.
Thus $gw_n \not\in \{ v_0, \ldots, v_{n-1} \}$.
Since $K$ is a pseudomanifold,
we must have $gw_n \in \{ v_+, v_- \}$, say $gw_n = v_+$.
Then $w^* = w^*_n = (gw_n)^* = v^*_+$,
Similarly, $w^* = v^*_-$ if $gw_n = v_-$.
This verifies property (PM2) for $K^*$. We have shown that
$K^*$ is an $n$-dimensional pseudomanifold.

To prove that $K^*$ is oriented, we need to orient the
manifold $U^*$ given by the complement of the polyhedron
$\Sigma^*$ of the $(n-2)$-skeleton of $K^*$.
Let $U$ be the manifold given by the complement of the polyhedron
$\Sigma$ of the $(n-2)$-skeleton of $K$.
Then $U$ is oriented by assumption.
Since the orbit projection $\pi: K \to K^*$ is simplicial and finite-to-one,
we have $\pi^{-1} (\Sigma^*) = \Sigma$ and
$\pi^{-1} (U^*) = U$. Thus $U$ is a $G$-invariant set and
$\pi$ restricts to a map $\pi: U \to U^*$, which is
the orbit projection $U \to U/G = U^*$.
Now whenever a finite group acts in an orientation preserving way on
an oriented manifold, then the orbit space is an orientable 
rational homology manifold. So $U/G$, and hence $U^*$ is orientable.
\end{proof}
In this theorem, the assumption of orientedness of the action
is essential. Simple counterexamples, such as the complex conjugation
action of $\intg/_2$ on the unit circle, readily show that the
pseudomanifold property is not generally preserved by nonoriented actions.
If one allows strata of codimension $1$ (i.e. if one drops requirement (PM2))
and only retains the density condition, expressed simplicially in (PM1),
then Popper shows in \cite[Thm. 3.4]{popper}
that the orbit space of a continuous action of a compact Lie group on a topological
pseudomanifold whose orbits have conical slices is again a topological pseudomanifold.
For our purposes, the classical notion of pseudomanifold as understood
by Goresky-MacPherson \cite{gmih1} is the relevant one.\\

Since the orbit projection associated to the action of a finite group
is a ramified covering (Example \ref{exple.orbitprojisramifiedcov}), 
we obtain, using Theorem \ref{thm.orbitspacepseudomanifold}, 
the following corollary to Theorem \ref{thm.baseoframifiedcoveriswitt}:
\begin{cor} \label{cor.orbitspaceiswitt}
Let $G$ be a finite group and $X$ an oriented triangulated Witt pseudomanifold
upon which $G$ acts by orientation preserving simplicial maps.
Then the orbit space $X/G$ is an oriented Witt pseudomanifold.
\end{cor}

\section{$G$-Signatures of Singular Spaces}
\label{sec.gsignatures}

Let $G$ be a finite group which acts (from the left)
simplicially on an oriented pseudomanifold $X$
of even dimension $n=2m$. We assume that this action preserves the
orientation of $X$ and that $X$ satisfies the Witt condition.
To define a $G$-signature for $X$, we proceed along the lines of
Atiyah-Singer \cite[\S 6]{atiyahsinger}, except that
the singularities in $X$ require us to work with intersection homology instead of
ordinary homology.
Let $IH_* (X)$ denote the intersection homology of $X$, with real coefficients, taken
with respect to the lower or upper middle perversity. The Witt condition ensures
that the canonical map from lower to upper middle perversity intersection 
homology is an isomorphism, and that the bilinear intersection form
\[ B: IH_m (X) \times IH_m (X) \longrightarrow \real \]
is nondegenerate. This form is symmetric when $m$ is even and
skew-symmetric when $m$ is odd.
For every $g\in G$,
the map $g: X\to X$ is a simplicial isomorphism and thus
induces an automorphism $g_*: IH_m (X)\to IH_m (X)$, which
endows $IH_m (X)$ with a left $G$-module structure.
When $X$ is a rational homology manifold,
the action of $G$ is usually
(e.g. in \cite{atiyahsinger} and \cite{zagier}) 
considered on \emph{co}homology $H^m (X)$ rather than homology.
Since cohomology is contravariant, $H^m (X)$ is made into a \emph{left} 
$G$-module by $g\cdot x := (g^*)^{-1} (x)$
(see also Hirzebruch-Zagier \cite[p. 30]{hirzebruchzagier} 
and Gilmer \cite[p. 106]{gilmer}).
Then the Poincar\'e duality isomorphism $H^m (X) \cong H_m (X)$
is an isomorphism of left $G$-modules in view of
\[ 
 g^{-1*} (\xi) \cap [X]
= g_* g^{-1}_* (g^{-1*} (\xi) \cap [X])
= g_* (\xi \cap g^{-1}_* [X])
= g_* (\xi \cap [X]).
\]
In any case, the traces of $g_*$ and $g^{-1}_*$ are equal for a real representation
of a finite group, but in the complex case, which is relevant when $m$ is odd, the
two traces are only conjugate to each other.
We also note that the Poincar\'e duality isomorphism takes the cohomological
intersection form to the homological form $B$.
Consideration of the $G$-signature involves only the perversity $\bm$-intersection
homology in the middle degree $m$, and thus no connectivity requirement
on $X$ is necessary when applying the results of 
Section \ref{sec.ramcovsandihtransfer}.
The form $B$ is $G$-invariant, since the action of $G$ on $X$ preserves the orientation.
Let $\langle \cdot, \cdot \rangle$ be a positive definite and $G$-invariant 
inner product on $IH_m (X)$.
An operator $A$ is defined by the equation 
$B(x,y) = \langle x, Ay  \rangle$.
This operator commutes with the action of $G$, and
its adjoint is given by $A^* = (-1)^m A$.
Since $A$ is $G$-equivariant, $G$ acts on each eigenspace of $A$.

Suppose that $m$ is even.
Then $A$ is self-adjoint and $IH_m (X)$ decomposes as a direct sum
$IH_m (X) = IH_+ \oplus IH_-,$ where $IH_+$ 
is the direct sum of the eigenspaces of the positive eigenvalues of $A$ and 
$IH_-$ is the direct sum of the eigenspaces of the negative eigenvalues.
These are $G$-invariant and
thus define real $G$-modules that we will also denote by
$IH_+, IH_-$. Up to isomorphism, they are independent of the choice
of inner product $\langle \cdot, \cdot \rangle$, as $G$ 
has discrete characters, the characters of $IH_+$ and $IH_-$
vary continuously with the inner product, and
the space of $G$-invariant positive definite 
inner products is connected.
If $V$ is a finite dimensional real vector space and $B$ is a 
nondegenerate symmetric bilinear pairing on $V$, 
then a subspace $V_+ \subset V$ is called
\emph{positive} if $B|_{V_+ \otimes V_+}$ is positive
definite, i.e. $B(v,v)>0$ for all $v\in V_+$,
$v\not= 0$. The subspace $V_+$ is called \emph{maximally positive},
if $V_+$ is not contained in a larger positive subspace. 
Negative and maximally negative subspaces are defined analogously.
\begin{lemma} \label{lem.decompimpmaximalpos}
Let $V_+ \subset V$ be a positive subspace for $B$ and $V_- \subset V$ 
a negative subspace.
If $V= V_+ \oplus V_-$, then $V_+$ is maximally positive and
$V_-$ is maximally negative.
\end{lemma}
This is well-known and straightforward.
Sylvester's law of inertia states that any two maximally positive (respectively, negative)
subspaces have the same dimension. The signature of $B$ is the difference
$\dim V_+ - \dim V_-$, for any maximal positive subspace $V_+$ and any
maximal negative $V_-$.
\begin{lemma} \label{lem.ihpmaxposforb}
The subspace $IH_+$ is maximally positive for the intersection form $B$.
The subspace $IH_-$ is maximally negative for $B$. 
\end{lemma}
\begin{proof}
The spectral theorem applied to the self-adjoint operator $A$
can be used to show that $IH_+$ is positive for $B$.
Similarly, $IH_-$ is negative.
Since $IH_m (X) = IH_+ \oplus IH_-$, Lemma \ref{lem.decompimpmaximalpos}
implies that $IH_+$ is in fact maximally positive and
$IH_-$ is maximally negative.
\end{proof}

For the following definition see also \cite{blp}. Let
$RO (G)$ denote the real, and $R(G)$ the complex, representation ring of $G$.
\begin{defn}
For $m$ even, the \emph{$G$-signature} of the 
$G$-Witt space $X^{2m}$ is the virtual representation
\[ \Sign (G,X) := IH_+ - IH_- \in RO(G) \subset R(G).  \]
\end{defn}
\noindent On elements $g\in G,$ we will in particular consider the real numbers
\[  \Sign (g,X) := \tr(g|_{IH_+}) - \tr(g|_{IH_-}), \]
where $\tr$ denotes the trace of a linear endomorphism.
This number depends only on the action of $g$ on $IH_* (X)$.

If $m$ is odd, then $A$ is skew-adjoint.
Let $(AA^*)^{1/2}$ denote the positive square root of $AA^* = -A^2$.
Since the square of the operator $J = A/ (AA^*)^{1/2}$ 
is $J^2 = -1$, $J$ defines a complex structure on $IH_m (X)$.
As $J$ commutes with the $G$-action, we obtain a complex $G$-module $IH_m (X)$.
Again, this module is independent of the choice of inner product.

\begin{defn}
For $m$ odd, the \emph{$G$-signature} of the 
$G$-Witt space $X^{2m}$ is the virtual representation
\[ \Sign (G,X) := IH_m (X) - IH_m (X)^* \in R(G).  \]
\end{defn}
\noindent On elements $g\in G,$ we will in particular consider the complex numbers
\[  \Sign (g,X) := \tr (g|_{IH_m (X)}) - \overline{\tr (g|_{IH_m (X)})} 
  = 2i~ \operatorname{Im} \tr(g|_{IH_m (X)}), \]
where one takes the trace of $g$ as an automorphism of a 
complex vector space.
This number is again independent of choices.

\begin{remark} \label{rem.gsignforgequalone}
For $g=1$, one obtains the ordinary signature $\Sign (X)$.
Indeed, if $m$ is even, then
\[ \Sign (1,X)
    = \tr(\id|_{IH_+}) - \tr(\id|_{IH_-})
    = \dim IH_+ - \dim IH_-
    = \Sign (X), \]
using Lemma \ref{lem.ihpmaxposforb}.    
If $m$ is odd, then
\[ \Sign (1,X)
    = 2i~ \operatorname{Im} \tr(\id|_{IH_m (X)})
    = 2i~ \operatorname{Im} \dim_\cplx IH_m (X)
    = 0 = \Sign (X). \]
\end{remark}
If $X$ is odd dimensional, one sets $\Sign (g,X)=0$.
The next lemma is a consequence of the topological invariance of
the intersection form $B$ 
(see e.g. \cite[Remark 8.2.7 and Thm. 9.3.16]{friedmanihbook}).
\begin{lemma} \label{lem.gsignpresequivhomeo}
Let $X$ and $Y$ be oriented closed $G$-Witt pseudomanifolds.  
If $X$ and $Y$ are $G$-equivariantly and orientation preservingly
homeomorphic, then
$\Sign (g,X) = \Sign (g,Y)$ for every $g\in G$.
\end{lemma}

\begin{prop} \label{prop.equivsignconjugation}
Let $X$ be a closed oriented $G$-Witt pseudomanifold such that the $G$-action
preserves the orientation. Then
$\Sign (hgh^{-1}, X) =  \Sign (g, X)$
for all $g,h\in G$.
\end{prop}
\begin{proof}
We consider the automorphism $h: X \to X$ given by $h(x) = h\cdot x$.
This is generally not a $G$-equivariant map.
We make it $G$-equivariant by defining a new $G$-action 
$\bullet$ on $X$: 
For $x\in X$, we set $g\bullet x := hgh^{-1} \cdot x$.
Then $h: (X,\cdot) \to (X,\bullet)=:Y$ is indeed $G$-equivariant, as
\[ h (g\cdot x) = h\cdot (g \cdot x)
= hgh^{-1} \cdot (h\cdot x) = g\bullet (h(x)). \]
The action $\bullet$ is again simplicial and preserves the orientation of $X$,
since $x \mapsto g\bullet x$ is the composition of the 
homeomorphisms 
$x \mapsto h^{-1} \cdot x$,
$x \mapsto g\cdot x$, and
$x \mapsto h\cdot x$,
each of which is simplicial and preserves the orientation of $X$ by assumption.
By Lemma \ref{lem.gsignpresequivhomeo},
$\Sign (g, (X,\cdot)) = \Sign (g, (X,\bullet)) = \Sign (hgh^{-1}, (X,\cdot))$.
\end{proof}

The proof of Proposition \ref{prop.signatureoforbitspace}  on the signature of
the orbit space requires 
the following standard fact from the representation theory of finite groups;
see e.g. Fulton-Harris \cite[p. 15f]{fultonharris},
or Hirzebruch-Zagier \cite[Thm. p. 21f]{hirzebruchzagier}.
\begin{lemma} \label{lem.reptheoryfingroups}
Let $G$ be a finite group acting linearly on a 
finite dimensional real or complex vector space $V$.
Then the dimension of the linear subspace 
$V^G = \{ v\in V ~|~ gv=v \text{ } \forall g\in G \}$ of invariant vectors
is given by
\[ \dim V^G = \frac{1}{|G|} \sum_{g\in G} \tr (g|_V). \]
\end{lemma}

Let $\pi: X \to X/G$ be the orbit projection, a simplicial map.
By Corollary \ref{cor.orbitspaceiswitt}, $X/G$ is an oriented Witt pseudomanifold 
of dimension $2m$,
and thus has a well-defined signature $\Sign (X/G)$, which can be described
as follows:
Since $X/G$ is Witt,
the canonical map from lower to upper middle perversity intersection 
homology is an isomorphism, and the bilinear intersection form
$\overline{B}: IH_m (X/G) \times IH_m (X/G) \to \real$
is nondegenerate. This form is symmetric when $m$ is even and
skew-symmetric when $m$ is odd. 

Suppose that $m$ is even.
Using the map $\pi_*: IH_m (X) \to IH_m (X/G)$, we define linear subspaces
of $IH_m (X/G)$ by
$\overline{IH}_+ := \pi_* (IH_+),$
$\overline{IH}_- := \pi_* (IH_-).$

\begin{lemma} \label{lem.transfoihisihcapinv}
The image $\pi_! (\overline{IH}_\pm)$ consists precisely of the $G$-invariant
elements in $IH_\pm$, that is,
$\pi_! (\overline{IH}_+) = IH_+ \cap IH_m (X)^G,$
   $\pi_! (\overline{IH}_-) = IH_- \cap IH_m (X)^G.$
\end{lemma}
\begin{proof}
For $w\in IH_+,$ one has $\pi_! \pi_* (w)= \sum_g g_* (w)$.
The elements $g_* (w)$ lie in $IH_+$ for every $g\in G$, 
as $IH_+$ is $G$-invariant.
Thus their sum, and hence $\pi_! \pi_* (w)$ is in $IH_+$.
Since every element of $\overline{IH}_+$ is of the form $\pi_* (w)$,
with $w\in IH_+$, this shows that $\pi_! (\overline{IH}_+) \subset IH_+$.
By Proposition \ref{prop.rattransferisisoontoinvariants},
$\pi_! (\overline{IH}_+) \subset \pi_! (IH_m (X/G))
   \subset IH_m (X)^G.$
We have proved that $\pi_! (\overline{IH}_+) \subset IH_+ \cap IH_m (X)^G$.
Conversely, let $w \in IH_+ \cap IH_m (X)^G$ be any element.
Set $d:=|G|$ and $v := \frac{1}{d} \pi_* (w)$. 
Then $v\in \overline{IH}_+$ and
\[ \pi_! (v) = \mbox{$\frac{1}{d}$} \pi_! \pi_* (w)
   =  \mbox{$\frac{1}{d}$} \sum_{g\in G} g_* (w) =    \mbox{$\frac{1}{d}$} dw =w, \]
since $w\in IH_m (X)^G$ is $G$-invariant so that $g_* (w)=w$ for all $g$.
We conclude that $w\in \pi_! (\overline{IH}_+)$, which 
establishes the converse inclusion.
\end{proof}

\begin{lemma} \label{lem.oihpispos}
The subspace $\overline{IH}_+$ is positive for the 
intersection form $\overline{B}$.
The subspace $\overline{IH}_-$ is negative for $\overline{B}$. 
\end{lemma}
\begin{proof}
Given a nonzero vector $v \in \overline{IH}_+$, we must prove that
$\overline{B}(v,v)>0$.
Using the up-down Proposition \ref{prop.updownmultbydegree}, we can write
$v$ as $v = \frac{1}{d} \pi_* \pi_! (v).$
Using the adjointness relation
$B (\pi_! (-), -) = \overline{B} (-, \pi_* (-))$ (\cite[p. 390, 3.]{gmlefschetz})
and up-down,
\begin{align*}  
\overline{B} (v,v)
&= \mbox{$\frac{1}{d^2}$} \overline{B} (\pi_* \pi_! v, \pi_* \pi_! v) 
 = \mbox{$\frac{1}{d^2}$} B (\pi_! \pi_* \pi_! v, \pi_! v) \\
&= \mbox{$\frac{1}{d^2}$} B (\pi_! (dv), \pi_! v) 
 = \mbox{$\frac{1}{d}$} B (\pi_! v, \pi_! v).
\end{align*}   
According to Lemma \ref{lem.transfoihisihcapinv}, 
the vector $\pi_! v$ is in $IH_+$, and it is nonzero since
$\pi_!$ is injective.
By Lemma \ref{lem.ihpmaxposforb}, $IH_+$ is positive for $B$.
Therefore, $B (\pi_! v, \pi_! v) >0$.
\end{proof}

\begin{lemma} \label{lem.ihxmodgdecompoihpoihn}
The vector space $IH_m (X/G)$ decomposes as a direct sum
$IH_m (X/G) = \overline{IH}_+ \oplus \overline{IH}_-.$
\end{lemma}
\begin{proof}
Given any $v\in IH_m (X/G)$, we put
$w:= \pi_! (v) \in IH_m (X)^G \subset IH_m (X)$
(Proposition \ref{prop.rattransferisisoontoinvariants}).
Using the decomposition $IH_m (X) = IH_+ \oplus IH_-$,
there are $w_+ \in IH_+,$ $w_- \in IH_-$ with $w = w_+ + w_-$.
It follows that
$v = \frac{1}{d} \pi_* \pi_! (v)
     = \frac{1}{d} \pi_* (w)
     = \frac{1}{d} \pi_* (w_+) +  \frac{1}{d} \pi_* (w_-),$
where the first summand is in $\overline{IH}_+$ and the
second one in $\overline{IH}_-$.
It remains to prove that $\overline{IH}_+ \cap \overline{IH}_- = \{ 0\}.$
Suppose that $v$ is a vector in this intersection.
If $v$ were nonzero, then by Lemma \ref{lem.oihpispos},
$\overline{B} (v,v)>0$ since $v\in \overline{IH}_+$ and
$\overline{B} (v,v)<0$ since $v\in \overline{IH}_-$.
This contradiction shows that $v=0$.
\end{proof}

Using Lemma \ref{lem.decompimpmaximalpos},
we deduce from Lemmas \ref{lem.oihpispos} and \ref{lem.ihxmodgdecompoihpoihn} 
that $\overline{IH}_+$ is in fact maximal positive for $\overline{B},$
and $\overline{IH}_-$ is maximal negative.
In particular,
$\Sign (X/G) = \dim \overline{IH}_+ - \dim \overline{IH}_-.$
If $m$ is odd, then $\overline{B}$ is skew-symmetric
and hence $\Sign (X/G)=0$.

\begin{prop} \label{prop.signatureoforbitspace}
The signature of the orbit space is given in terms of equivariant signatures
by
\[ \Sign (X/G) = \frac{1}{|G|} \sum_{g\in G} \Sign (g,X). \]
\end{prop}
\begin{proof}
Suppose first that $m$ is even.
The transfer isomorphism of Proposition \ref{prop.rattransferisisoontoinvariants})
\[ \pi_!: IH_m (X/G) \stackrel{\simeq}{\longrightarrow} IH_m (X)^G  \]
restricts to a monomorphism
$\pi_!: \overline{IH}_+ \hookrightarrow IH_m (X)^G,$
whose image is $IH_+ \cap IH_m (X)^G = (IH_+)^G$ 
by Lemma \ref{lem.transfoihisihcapinv}.
Hence the restriction is an isomorphism
$\pi_!: \overline{IH}_+ \stackrel{\simeq}{\longrightarrow} (IH_+)^G.$
The restriction to the negative subspace similarly yields an isomorphism
$\pi_!: \overline{IH}_- \stackrel{\simeq}{\longrightarrow} (IH_-)^G.$
Thus, using Lemma \ref{lem.reptheoryfingroups},
\begin{align*}
\Sign (X/G)
&= \dim \overline{IH}_+ - \dim \overline{IH}_- 
 = \dim (IH_+)^G - \dim (IH_-)^G \\
&= \frac{1}{|G|} \sum_{g\in G} \tr (g_*|_{IH_+}) - 
       \frac{1}{|G|} \sum_{g\in G} \tr (g_*|_{IH_-}) \\
&= \frac{1}{|G|} \sum_{g\in G} \left( \tr (g_*|_{IH_+}) - 
         \tr (g_*|_{IH_-}) \right) 
 = \frac{1}{|G|} \sum_{g\in G} \Sign(g,X).               
\end{align*}

Now assume that $m$ is odd.
In this case, $\Sign (X/G)=0$. Again using Lemma \ref{lem.reptheoryfingroups},
this agrees with
\begin{align*}
\frac{1}{|G|} \sum_{g\in G} \Sign(g,X)
  &= \frac{1}{|G|} \sum_{g\in G}  
     2i~ \operatorname{Im} \tr(g_*|_{IH_m (X)}) 
   = \frac{2i}{|G|} \operatorname{Im} \sum_{g\in G}  
         \tr(g_*|_{IH_m (X)}) \\
  &= 2i \operatorname{Im} \dim (IH_m (X)^G) 
   = 0.
\end{align*}
\end{proof}

\section{Equivariant Zagier-$L$-Classes}
\label{sec.equivariantzagierlclass}

Let $X$ be an oriented closed PL pseudomanifold of dimension $n$
that satisfies the Witt condition. Suppose that a finite group 
$G$ acts on $X$.
The action is to be simplicial with respect to some triangulation
and preserves the orientation.
For any $g\in G$, we shall define equivariant $L$-classes
\[ L_* (g,X) \in H_* (X; \cplx).\]
In the special case where $X$ is a rational homology manifold, these
classes have been defined by Zagier in \cite{zagier}.
In the smooth case, Zagier's classes, and hence the classes defined in this
paper, agree under a Gysin homomorphism with the Poincar\'e duals of
the equivariant Atiyah-Singer classes introduced in \cite{atiyahsinger}.

We triangulate the oriented $i$-sphere $S^i$ by the complex
$C$ given as the boundary of a standard $(i+1)$-simplex.
We endow $S^i$ with the trivial $G$-action.
Let $\overline{f}: X/G \to S^i$ be a continuous map.
We approximate it by a simplicial map as follows:
Let $K$ be the (finite) simplicial complex with polyhedron $|K|=X$ such that
$G$ acts simplicially on $K$.
As pointed out in Section \ref{sec.simplicialactions},
we may assume that this action is regular so that the
complex $L:= K/G$ triangulates $X/G$ and the orbit projection $\pi: K \to L$ is
simplicial.
By the finite simplicial approximation theorem,
there exists an $N$ such that $\overline{f}$ has a simplicial approximation
$\bsd^N (L) \to C$.
(Here, $\bsd^N$ denotes the $N$-th abstract barycentric subdivision. Note that
the complex $C$ does not have to be subdivided.)
We denote this simplicial approximation again by $\overline{f}$.
The simplicial map $\pi: K\to L$ induces a simplicial map
$\pi: \bsd^N (K) \to \bsd^N (L)$.
(Here, \emph{abstract} barycentric subdivision is relevant, since a simplicial
map will not generally map \emph{geometric} barycenters to 
\emph{geometric} barycenters.)
Any simplicial action induces a simplicial action on barycentric subdivision.
Hence, $G$ still acts simplicially on $\bsd^N (K)$.
Composing, we obtain a $G$-invariant map $f$:
 \[ \xymatrix@R=15pt{
 X \ar[r]^\pi \ar[rd]_{f} & X/G \ar[d]^{\overline{f}} \\
 & S^i.
 } \]
Let $\Delta^\circ \in C$ be the interior of 
an $i$-simplex and let $p\in \Delta^\circ$ be any point,
viewed as the (abstract) barycenter of $\Delta$.
Then the preimage
$f^{-1} (p) \subset X$ is a $G$-invariant subspace.
We give $f^{-1} (p)$ the structure of a simplicial subcomplex as follows:
The point $p$ is a vertex of $\bsd (C)$.
Then the diagram of simplicial maps
\[ \xymatrix@R=15pt{
 \bsd^{N+1} (K) \ar[r]^\pi \ar[rd]_{f} & \bsd^{N+1} (L) \ar[d]^{\overline{f}} \\
 & \bsd (C)
 } \]
shows that $f^{-1} (p)$ is a subcomplex of $\bsd^{N+1} (K)$.
Since $G$ acts simplicially on $\bsd^{N+1} (K)$ and $f^{-1} (p)$ is a 
$G$-invariant subcomplex, the restricted action of $G$ on
$f^{-1} (p)$ is still simplicial.
The preimage is compact since it is closed in the compact space $X$.
It is contained in the open neighborhood $f^{-1} (\Delta^\circ)$.
Let $\phi: f^{-1} (\Delta^\circ) \to \Delta^\circ \times f^{-1} (p)$
be the PL homeomorphism provided by
Milnor-Stasheff \cite[p. 236, Lemma 20.5]{milnorstasheff}.
(See also Curran \cite[p. 121f]{curran}, as well as,
in the equivariant case, Zagier \cite[p. 22]{zagier}.) 
This homeomorphism is such that $f\circ \phi^{-1}$ is
the first factor projection to $\Delta^\circ$.
Furthermore, $\phi$ restricts to the identity map on 
$f^{-1} (p) \subset f^{-1} (\Delta^\circ)$.
This shows that $f^{-1} (p)$ is normally nonsingular in $X$, with 
trivial normal bundle.
As $f^{-1} (\Delta^\circ) = \pi^{-1} (\overline{f}^{-1} (\Delta^\circ))$, the tube is 
$G$-invariant.
We endow $\Delta^\circ \times f^{-1} (p)$ with the $G$-action
that makes $\phi$ $G$-equivariant.
Let $(q,x) \in \Delta^\circ \times f^{-1} (p)$ be a point and $g\in G$.
The translate $g(q,x)$ has the form $(q',x')$.
As $f$ is $G$-invariant,
\[ q' = f \circ \phi^{-1} (q',x') = f \circ \phi^{-1} (g (q,x)) 
    =  f(g\cdot \phi^{-1} (q,x))
    = f \circ \phi^{-1} (q,x) =q. \]
Thus $g(q,x) = (q,x'),$ i.e. $g$ acts trivially on the first component.
Now, the PL pseudomanifold property desuspends, i.e. if
$A\times \real^i$ is a PL pseudomanifold, then $A$ is a pseudomanifold.
(The reason is that the pseudomanifold property can be checked with respect
to any PL stratification. Now use the PL intrinsic stratification. For this
stratification, the intrinsic strata of $A\times \real^i$ are the products
with $\real^i$ of the intrinsic strata of $A$, see \cite[Prop. A.1]{blm}.)
Similarly, the Witt condition desuspends, as is shown in
\cite[p. 46, Lemma 14.1]{banaglborelequivlclass}.
Therefore, $f^{-1} (p)$ is a pseudomanifold satisfying the Witt condition;
its dimension is $n-i$.
The orientation of $S^i$ induces an orientation of $\Delta^\circ$.
Together with the orientation of $X$, it induces an orientation of
$f^{-1} (p)$.
Let $\pi^i (X)$ denote the Borsuk-Spanier cohomotopy set of $X$, consisting of
pointed homotopy classes of continuous maps $X\to S^i$. 
Suppose now that $n \leq 2(i-1)$.
Then $\pi^i (X)$ is an abelian group. 
There is a natural action of $G$ on $\pi^i (X)$ and the subgroup
of invariant elements will be denoted by $\pi^i (X)^G$.
For $g\in G$, we define a map
\[ S_p (g): \pi^i (X)^G \otimes \cplx \to \cplx \]
by
\[ S_p (g)([f] \otimes \lambda) :=
    \lambda \cdot \Sign (g, f^{-1} (p)), \]
using the $G$-signature of $G$-Witt pseudomanifolds as
described in Section \ref{sec.gsignatures} and in
\cite{blp}.
If $p'$ is any other point in $\Delta^\circ$, then 
$\phi$ restricts to an orientation preserving $G$-homeomorphism
$\phi|: f^{-1} (p') \cong \{ p' \} \times f^{-1}(p)$ and thus
$S_p (g)([f] \otimes \lambda) = S_{p'} (g)([f] \otimes \lambda)$
by Lemma \ref{lem.gsignpresequivhomeo}.
Suppose that $F: X \times I \to S^i$ is a 
$G$-equivariant homotopy from $f$ to $f'$.
Since $G$ acts trivially on $S^i$, $F$ factors through a homotopy
$\overline{F}: (X/G) \times I \to S^i$, which we may take to
be simplicial, as described above. Composing with $\pi$, we
receive a $G$-invariant simplicial version of the homotopy $F$ such that
the preimage $F^{-1} (p)$ is a $G$-invariant subcomplex of a simplicial
subdivision of $K \times I$ and a simplicial $G$-space.
Since $F^{-1} (p)$ is also a Witt bordism from $f^{-1} (p)$ to
$f'^{-1} (p)$, we obtain a simplicial Witt $G$-bordism
from $f^{-1} (p)$ to  $f'^{-1} (p)$.
We now use the invariance of $G$-signatures under Witt $G$-bordisms,
a result due to Eric Leichtnam, Paolo Piazza and the author.
\begin{prop} \label{prop.wittgbordisminvariance}
(Banagl-Leichtnam-Piazza.) The equivariant signature $\Sign (g,X)$
of closed oriented Witt pseudomanifolds $X$ is a simplicial Witt $G$-bordism invariant.
\end{prop}
\begin{proof}
In the nonsingular case, this was shown by Ossa in \cite{ossa}.
In the singular Witt situation, this follows from the equivariant 
signature operator methods of \cite{blp}. In more detail, one argues as follows.
Let $W^{n+1}$ be a simplicial Witt $G$-bordism 
with $\partial W = X \sqcup -X',$ $n$ even.
Then $W$ can in particular be smoothly stratified by taking the open simplices
as strata, see also \cite[p. 37, Prop. 4.5]{abp}.
Since $G$ acts linearly on simplices, it acts in particular by 
stratified diffeomorphisms.
There exists a $G$-invariant wedge metric on $W$, and it can be taken
to be a product near the boundary.
Then the signature operator $D_W$ 
(see Albin, Leichtnam, Mazzeo, Piazza \cite{almpsigpack}) 
with respect to such a metric is $G$-invariant and thus defines a class
$[D_W] \in K^G_{n+1} (W, \partial W)$ in
$G$-equivariant analytic relative K-homology
such that
$\partial_* [D_W] = [D_X] - [D_{X'}]$ (up to a multiple of $2$).
The constant maps $C: W\to \pt,$ $c: X\to \pt,$ $c': X'\to \pt$
(and indeed more general $G$-maps $W\to Y$)
fit into a commutative diagram
\[ \xymatrix@R=15pt{
K^G_{n+1} (W,\partial W) \ar[d]_{\partial_*} \ar[r]^{C_*} 
   & K^G_{n+1} (\pt, \pt) =0 \ar[d]^{\partial_*} \\
K^G_n (X) \oplus K^G_n (X') \ar[r]_>>>>>>>>{c_* + c'_*} & K^G_n (\pt) 
} \]
It follows that $c_* [D_X] = c'_* [D_{X'}].$
The pushforward $c_* [D_X] \in K^G_n (\pt) = R(G)$ under the constant map
$c: X \to \pt$ to a point is the equivariant index of $D_X$.
At an element $g\in G$, this index is nothing but $\Sign (g,X)$,
see \cite{blp} for more information.
\end{proof}
By the proposition, $S_p (g)[f] = S_p (g)[f']$.
The map $f$ is equivariantly homotopic to $s\circ f$ for any
simplicial automorphism $s$ of $C$.
Thus we may move $\Delta^\circ$ to any other $i$-simplex of $C$
and $S_p (g)[f]$ does not change.
So altogether then, $S_p (g)[f] = S_{p'} (g)[f']$
for $G$-homotopic maps $f,f'$ and almost all $p,p'\in S^i$ (not necessarily in the
interior of the same $i$-simplex).
This shows that $S_p (g)$ is well-defined and independent of $p$ for
almost all $p$.
We may thus simply write $S(g)$ for $S_p (g)$.
The map $S(g)$ is linear.
Hurewicz maps provide vertical homomorphisms such that the diagram
 \[ \xymatrix@R=20pt{ 
 H^i (X/G;\cplx) \ar[r]^\simeq_{\pi^*} & H^i (X;\cplx)^G \\
 \pi^i (X/G) \otimes \cplx \ar[u]^\simeq \ar[r]^\simeq_{\pi^*} & 
 \pi^i (X)^G \otimes \cplx \ar[u]_\simeq
} \] 
commutes. 
In the range $n\leq 2(i-1)$, a well-known theorem of Serre (\cite{serre}) asserts
that the Hurewicz maps are isomorphisms.
Under these isomorphisms, $S(g)$ defines an element
\[ S(g) \in \Hom (H^i (X/G; \cplx), \cplx) = H_i (X/G;\cplx). \]
Its transfer is an invariant element $\pi_! S(g) \in H_i (X;\cplx)^G$.  
To compensate for the difference between $\pi_* [X]$ and
$[X/G]$ in $H_* (X/G),$ we divide by $\deg \pi$.
(If $G$ acts effectively, then $\deg \pi = |G|$, the order of $G$.)  
  
\begin{defn}  
The \emph{Atiyah-Singer-Zagier equivariant L-class} of
the closed oriented $G$-Witt pseudomanifold $X$
is on an element $g\in G$ defined to be
 \[ L_* (g,X) := \frac{1}{\deg \pi}  \pi_! S(g) \in H_* (X;\cplx)^G.  \]
\end{defn}

The restriction $n \leq 2(i-1)$ can be removed by taking
the product of $X$ with a $G$-invariant sphere of large dimension.

\begin{prop} \label{prop.evaloflgxissigngpreim}
Let $u\in H^i (S^i;\intg)$ be the generator compatible with the orientation
and let $\langle -,- \rangle$ denote evaluation of a cohomology class on
a homology class.
The equivariant class $L_* (g,X)$ satisfies
\begin{equation} \label{equ.fulgxissigngfinvp}
\langle f^* (u),~ L_* (g,X) \rangle = \Sign (g,f^{-1} (p)). 
\end{equation}
For $\dim X \leq 2(i-1),$
it is the unique $G$-invariant homology class satisfying this equation
for all $G$-invariant homotopy classes $f: X \to S^i$.
\end{prop}
\begin{proof}
The Hurewicz isomorphism sends $[\overline{f}]$ to
$\overline{f}^* (u)$.
Thus, by the construction of $S(g)$, we have
$\langle \overline{f}^* (u),~ S(g) \rangle 
  = \Sign (g, f^{-1} (p)).$
Using the up-down formula for the transfer $\pi_!$, we calculate
\begin{align*}  
\langle \overline{f}^* (u),~ S(g) \rangle
&= \langle \overline{f}^* (u),~ \frac{1}{\deg \pi} \pi_* \pi_! S(g) \rangle 
 = \langle \overline{f}^* (u),~ \pi_* L_* (g,X) \rangle \\  
&= \langle \pi^* \overline{f}^* (u),~ L_* (g,X) \rangle 
 = \langle f^* (u),~ L_* (g,X) \rangle. \\ 
\end{align*}  
\end{proof}

For $g=1$, Remark \ref{rem.gsignforgequalone} shows that
 $L_* (1,X) = L_* (X),$ the Goresky-MacPherson-Siegel class.

\begin{prop} \label{prop.equivlconjugation}
For any $g,h \in G$, $h_* L_* (g, X) = L_* (hgh^{-1}, X)$.
\end{prop}
\begin{proof}
We use notation as in the construction of $L_* (g,X)$.
Thus $f:X \to S^i$ is a simplicial $G$-invariant map,
$f^{-1} (p) \subset X$ is a $G$-invariant simplicial subcomplex,
$p\in \Delta^\circ \subset S^i$, 
and there is a $G$-equivariant PL homeomorphism
$f^{-1} (\Delta^\circ) \cong \Delta^\circ \times f^{-1} (p)$
under which $f$ looks like the factor projection to $\Delta^\circ$.
It suffices to verify that the equality of evaluations
\[  \langle f^* (u),~ h_* L_* (g,X) \rangle   
   = \langle f^* (u),~ L_* (hgh^{-1}, X) \rangle  \]
holds.
Using Proposition \ref{prop.evaloflgxissigngpreim},
the left hand side is the equivariant signature
\begin{align*}
 \langle f^* (u),~ h_* L_* (g,X) \rangle
&=  \langle h^* f^* (u),~ L_* (g,X) \rangle 
 =  \langle (f\circ h)^* (u),~ L_* (g,X) \rangle \\
&= \Sign (g, h^{-1} (f^{-1} (p))),
\end{align*}
while the right hand side is the equivariant signature
\[
 \langle f^* (u),~ L_* (hgh^{-1} ,X) \rangle
   = \Sign (hgh^{-1}, f^{-1} (p)),
\]
The space $Y = f^{-1} (p) \subset X$ is $G$-invariant and
thus $h^{-1} (Y)=Y$.
Hence by Proposition \ref{prop.equivsignconjugation},
$\Sign (g, h^{-1} (Y)) = \Sign (g,Y)
   = \Sign (hgh^{-1}, Y)$.
We note that $h$ does indeed preserve the orientation of $Y$, since
$h$ preserves the orientation of the open $G$-tube
$f^{-1} (\Delta^\circ) \cong \Delta^\circ \times f^{-1} (p)$
and it acts trivially in the normal direction $\Delta^\circ$,
as pointed out earlier.
\end{proof}

Suppose that $X$ is a rational homology manifold,
assumed to be oriented, closed and triangulated.
In that case, Zagier constructed a cohomology class
$L^* (g,X) \in H^* (X;\cplx),$ which is the invariant class uniquely determined
by
\[ \langle L^* (g,X) \cup f^* (u),~ [X] \rangle = \Sign (g,f^{-1} (p)) \]
for all simplicial $G$-invariant maps $f:X \to S^i$ with $p$ a sufficiently general
point as above, see \cite[p. 21, Thm. 1]{zagier}.
Comparing this equation with (\ref{equ.fulgxissigngfinvp}) above,
it follows that the equivariant L-class constructed in this paper is the
Poincar\'e dual of Zagier's cohomological class when $X$ is a rational
homology manifold. In that case, the cohomological expression
of Proposition \ref{prop.equivlconjugation} is $h^* L^* (g,X) = L^* (h^{-1} gh, X)$,
see \cite[p. 18 (16), p. 21 (3)]{zagier}.

Suppose that $X$ is a smooth closed oriented manifold such that
$G$ acts smoothly on $X$, preserving the orientation.
For an element $g\in G$, let $X^g$ be the submanifold of points fixed by $g$.
Atiyah and Singer constructed an equivariant class
in $H^* (X^g; \operatorname{or}_{X^g} \otimes \cplx)$,
explicitly given in terms of characteristic
classes of $X^g$ and of the equivariant normal bundle of the inclusion
$j: X^g \subset X$, see \cite[p. 582]{atiyahsinger}. 
We use the formulation $L' (g,X) \in H^* (X^g; \operatorname{or}_{X^g} \otimes \cplx)$ 
of these classes
as given by Zagier in \cite[p. 12, (27)]{zagier}.
The $G$-Signature Theorem of Atiyah and Singer is the relation
\[ \Sign (g,X) = \langle L' (g,X), [X^g] \rangle. \] 
The inclusion has an associated Gysin homomorphism
$j_!: H^* (X^g) \to H^* (X)$ and the image of the AS-class under $j_!$
is Zagier's class, i.e.
\[ j_! L'(g,X) = L^* (g,X), \]
\cite[p. 4, (3); p. 21, Thm. 1]{zagier}.
As the Gysin map corresponds under Poincar\'e duality to covariant
homological pushforward $j_*: H_* (X^g) \to H_* (X),$ the AS-class
is related to our homological class by
\[ j_* (L'(g,X) \cap [X^g]) = L_* (g,X)  \]
in the differentiable case.

\section{Application: The Goresky-MacPherson $L$-Class of the Orbit Space}
\label{sec.lclassorbitspace}

If a finite group $G$ acts orientation preservingly 
on a Witt pseudomanifold $X$, then we have seen earlier
(Corollary \ref{cor.orbitspaceiswitt})
that the orbit space $X/G$ is again a Witt pseudomanifold. 
Thus the $L$-class $L_* (X/G) \in H_* (X/G;\rat)$
is well-defined (\cite{gmih1}, \cite{siegel}, \cite{curran}). 
We now apply the equivariant $L$-classes
$L_* (g,X)$ constructed in Section \ref{sec.equivariantzagierlclass}
to compute $L_* (X/G)$, considered as an element of $H_* (X/G;\cplx)$.

\begin{thm} \label{thm.lclassorbitspace}
Let $G$ be a finite group and $X$ an oriented closed Witt pseudomanifold
upon which $G$ acts simplicially, preserving the orientation.
Then $X/G$ is a Witt pseudomanifold and its
Goresky-MacPherson-Siegel $L$-class $L_* (X/G)$ is related to the
equivariant $L$-classes by 
\begin{equation} \label{equ.lorbitspaceisaverage}
L_* (X/G) 
    = \frac{1}{|G|} \sum_{g\in G} \pi_* L_* (g,X), 
\end{equation}
where $\pi: X \to X/G$ is the orbit projection.
\end{thm}
\begin{proof}
The orbit space $X/G$ is an oriented compact Witt pseudomanifold by
Corollary \ref{cor.orbitspaceiswitt}. Therefore, it has a
well-defined $L$-class $L_* (X/G)$.
By the universal coefficient theorem for complex coefficients,
it suffices to show that
\begin{equation} \label{equ.vlorbitspaceasequivclasses}
\langle v,~ L_* (X/G) \rangle
    = \langle v,~ \pi_* \left( \frac{1}{|G|} \sum_{g\in G} L_* (g,X) 
      \right) \rangle 
\end{equation}      
for every $v\in H^i (X/G; \cplx)$.
Up to taking the product with a $G$-invariant large dimensional sphere,
it also suffices to work in the range $n \leq 2(i-1)$, where $n=\dim X$.
Some multiple $kv$ of $v$ can then be written as
$kv = \overline{f}^* (u)$ for some simplicial map 
$\overline{f}: X/G \to S^i,$ $u\in H^i (S^i)$ the generator. 
Let $f:X\to S^i$
be the $G$-invariant simplicial map $f=\overline{f} \circ \pi: X\to S^i$.
Then $\pi^* (kv) = f^* (u)$ and
the right hand side of (\ref{equ.vlorbitspaceasequivclasses}) is
\[
\langle v,~ \pi_* \left( \frac{1}{|G|} \sum_{g\in G} L_* (g,X) 
      \right) \rangle
 = \langle \pi^* v,~ \frac{1}{|G|} \sum_{g\in G} L_* (g,X) \rangle 
 = \frac{1}{k |G|} \sum_{g\in G}
  \langle  f^* u,~  L_* (g,X) \rangle.   
\]   
By Equation (\ref{equ.fulgxissigngfinvp}),
$\sum_{g}
  \langle  f^* u,~  L_* (g,X) \rangle
= \sum_{g}
   \Sign (g, f^{-1} (p)),$ and
by Proposition \ref{prop.signatureoforbitspace},
\[
\frac{1}{k |G|} \sum_{g\in G}
   \Sign (g, f^{-1} (p)) =
   \frac{1}{k} \Sign (f^{-1} (p) /G).
\]
Now, the orbit space of the preimage is
$f^{-1} (p) /G = (\pi^{-1} (\overline{f}^{-1} (p))) /G
  = \overline{f}^{-1} (p).$
Since $\overline{f}$ is transverse regular to $p$
(as $\overline{f}$ is simplicial and $p$ is in the interior of
a top dimensional simplex), it follows from the construction 
of the (nonequivariant) Goresky-MacPherson-Siegel $L$-class
(see also Curran \cite[p. 121f]{curran}) that
$\langle \overline{f}^* (u),~ L_* (X/G) \rangle
   = \Sign (\overline{f}^{-1} (p)).$
Thus
\begin{align*}
\langle v,~ \pi_* \left( \frac{1}{|G|} \sum_{g\in G} L_* (g,X) \right) \rangle
&= \frac{1}{k} \Sign (f^{-1} (p) /G) 
 = \frac{1}{k} \Sign (\overline{f}^{-1} (p)) \\
&= \frac{1}{k} \langle \overline{f}^* (u),~ L_* (X/G) \rangle 
 = \langle v,~ L_* (X/G) \rangle,
\end{align*}
as required.
\end{proof}

The left hand side of (\ref{equ.lorbitspaceisaverage})
comes from rational homology, while
the individual summands $L_* (g,X)$ on the right hand side
are only known to be elements in homology with complex coefficients.
Thus Theorem \ref{thm.lclassorbitspace} is in particular an integrality result 
on the equivariant classes $L_* (g,X)$.
Formula (\ref{equ.lorbitspaceisaverage})
corresponds to the formulae of
Cappell-Shaneson-Maxim-Sch\"urmann \cite[p. 1726f, (1.11), (1.12)]{cmssequivcharcl}
for the equivariant Hirzebruch class in the algebraic setting.

\section{Free Actions}
\label{sec.freeactions}

If $G$ acts freely, then the only contribution to $L_* (X/G)$
can arise from $L_* (1,X)$ in light
of the following result.

\begin{thm} \label{thm.freeaction}
(Free action.)
Suppose that $G$ acts freely on $X$.
If $g\not= 1$, then $L_* (g,X) =0.$
\end{thm}
\begin{proof}
By Proposition \ref{prop.evaloflgxissigngpreim}, it suffices to show
that $\Sign (g,X)=0$ for free $2m$-dimensional
$G$-Witt spaces $X$, $g\in G - \{ 1 \}$.
Since $\Sign (g,X)$ depends only on the action of the subgroup 
$\langle g \rangle \cong \intg/_k$ generated by $g\not= 1,$
we may as well assume $G = \langle g \rangle$.
Let $BG$ be the corresponding classifying space.

Suppose first that $X$ has the form $X= Y \times \intg/_k$,
where $G=\intg/_k$ acts trivially on the Witt space $Y$ and by
translation on the second factor.
Then $IH_m (X) = \bigoplus_{j=0}^{k-1} IH_m (Y \times g^j)$
and the intersection form $B$ on $IH_m (X)$ is given by an orthogonal sum
of $k$ copies of the intersection form $B_Y$ on $IH_m (Y)$.
Choose a positive definite inner product $\langle \cdot, \cdot \rangle$
on $IH_m (Y) = IH_m (Y\times g^0)$ and extend it to $IH_m (X)$
by 
$\langle (v, g^j), (w, g^j) \rangle = \langle (v,g^0), (w,g^0)  \rangle$
and $\langle (v, g^i), (w, g^j) \rangle = 0$ for $g^i \not= g^j.$
This is then a positive definite and $G$-invariant inner product
on $IH_m (X)$. Suppose that $m$ is even.
Let $IH_+ (Y)$ and $IH_- (Y)$ be the positive and negative eigenspaces
of the operator $A_Y$ defined on $IH_m (Y)$ by 
$B_Y (v,w) = \langle v, A_Y w \rangle$, yielding a decomposition
$IH_m (Y) = IH_+ (Y) \oplus IH_- (Y)$.
The operator $A$ on $IH_m (X)$
defined by $B (v,w) = \langle v, A w \rangle$
is a block diagonal sum of $k$ copies of $A_Y$. 
The positive and negative eigenspaces of $A = k\cdot A_Y$
are given by
$IH_\pm = \bigoplus_j IH_\pm (Y) \times g^j$.
Now let $\{ (e^+_1, g^0),\ldots,  (e^+_s, g^0) \}$ be a basis for $IH_+ (Y)$.
Then a basis for $IH_+$ is given by
$\{ (e^+_1, g^j),\ldots,  (e^+_s, g^j) ~|~ j=0,\ldots, k-1  \}.$
If $g_*|_{IH_+}$ is written as a matrix with respect to this basis,
then this matrix has zero blocks, corresponding to $g^j$, along the
diagonal since $g_*$ acts by sending a basis vector $(e^+_i, g^j)$ to 
the basis vector $(e^+_i, g^{j+1})$.
This implies that the trace $\tr (g_*|_{IH_+})$ is zero.
Similarly,  $\tr (g_*|_{IH_-})=0$. Thus the difference $\Sign (g,X)$
of these two traces is zero.
The case of odd $m$ is treated similarly:
The operators $A_Y$ and $A$ are skew-adjoint and
the complex structure $J = A (AA^*)^{-1/2}$ on $IH_m (X)$
is the block diagonal sum of $k$ copies of the complex structure
$J_Y = A_Y (A_Y A_Y^*)^{-1/2}$ on $IH_m (Y)$.
Thus there is a direct sum decomposition of complex vector spaces
$(IH_m (X), J) = \bigoplus_{j=0}^{k-1} (IH_m (Y \times g^j), J_Y).$
Let $\{ (e_1, g^0),\ldots,  (e_s, g^0) \}$ be a basis for the complex
vector space $(IH_+ (Y), J_Y)$.
Then a basis for the complex vector space $(IH_m (X), J)$ is given by
$\{ (e_1, g^j),\ldots,  (e_s, g^j) ~|~ j=0,\ldots, k-1  \}.$
If the complex linear automorphism
$g_*|_{(IH_m (X), J)}$ is written as a matrix with respect to this basis,
then this matrix again has zero blocks along the
diagonal since $g_*|_{(IH_m (X), J)}$ acts by sending a basis vector $(e_i, g^j)$ to 
the basis vector $(e_i, g^{j+1})$.
This implies that the trace $\tr (g_*|_{(IH_m (X), J)})$ is zero.
Thus $\Sign (g,X)=0$ also when $m$ is odd.

We show next that some multiple of a free simplicial $G=\intg/_k$-Witt pseudomanifold $X$
is simplicially Witt $G$-bordant to a $G$-Witt pseudomanifold of the above form 
$Y \times \intg/_k$. This will imply the desired statement by the
Witt $G$-bordism invariance of the $g$-signature,
Proposition \ref{prop.wittgbordisminvariance}.
If $G=\langle g \rangle$ acts freely on $X$, then the orbit map
$\pi: X \to X/G$ is a regular covering space classified by a map
$f: Y := X/G \to BG$. The orbit space $X/G$ is an oriented Witt pseudomanifold by
Corollary \ref{cor.orbitspaceiswitt}.
Hence, the map $f$ defines an element
$[f] \in \Omega^\Witt_* (B\intg/_k)$, where $\Omega^\Witt_* (-)$ denotes
Witt bordism theory.
We recall that a connected space $S$ is called (homotopy theoretically) 
\emph{simple} if $\pi_1 (S)$ is abelian and the action of
$\pi_1 (S)$ on $\pi_n (S)$ is trivial for every $n\geq 2$.
A simple space has a rational localization $S_\rat$, which is again a simple space.
Standard localization methods in stable homotopy theory show that
$\widetilde{E}_n (S_\rat) = \widetilde{E}_n (S) \otimes \rat$
for every spectrum $E$.
The classifying space $BG=B\intg/_k = K(\intg/_k, 1)$ is a simple
space, since $\intg/_k$ is abelian and $\pi_n (K(\intg/_k, 1))=0$ for all
$n\geq 2$. (Explicitly, $B\intg/_k$ is an infinite dimensional lens space
$L = L^\infty_k$.)
The localization $L_\rat$ is homotopy equivalent to a point, since
$\pi_1 (L) \otimes \rat = \intg/_k \otimes \rat =0$,
and for $n\geq 2$, $\pi_n (L) \otimes \rat =0$.
Given any spectrum $E$, we deduce that
$\widetilde{E}_n (L) \otimes \rat = \widetilde{E}_n (L_\rat) 
= \widetilde{E}_n (\pt) \otimes \rat =0.$
The long exact sequence of the pair $(L, \{ l_0 \}),$ where $l_0 \in L$ is a base point,
shows that the inclusion $j:\{ l_0 \} \hookrightarrow L$
induces an isomorphism
$E_n (\{ l_0 \}) \otimes \rat \cong
   E_n (L) \otimes \rat.$
Applying this to the spectrum $E=\MWITT$ representing Witt bordism theory,
we receive an isomorphism
$\Omega^\Witt_n (\pt) \otimes \rat \cong
   \Omega^\Witt_n (B\intg/_k) \otimes \rat.$
We return to the element
$[f] \in \Omega^\Witt_* (BG)$ represented by the classifying map $f:X/G \to BG$.
The above isomorphism shows that some multiple of $[f]$ is
Witt bordant over $BG$ to a map which factors as
\[ Y \stackrel{c}{\longrightarrow} 
    \pt \stackrel{j}{\hookrightarrow}
     BG, \]
where $Y$ is some closed oriented Witt pseudomanifold.     
Let $F:W \to BG$ be such a bordism.     
Let $EG \to BG$ be the universal principal $G$-bundle.
(Explicitly, $EG = S^\infty$.) This is the universal cover of $BG=K(\intg/_k, 1)$.
Since $f$ classifies the $\intg/_k$-space $X$, we have
$X = f^* EG$. 
Let $V := F^* EG$ and $X' := c^* j^* EG$.
Then $V$ is a covering space of $W$ and thus also a Witt pseudomanifold.
The group $G=\intg/_k$ is the deck transformation group and in particular
operates freely on $V$.
Let $C$ be any simplicial complex triangulating $W$ which is fine enough 
so that every simplex
is contained in a connected open set which is evenly covered by $V\to W$.
Then $C$ can be lifted to a triangulation of $V$ by a simplicial complex $K$
such that $V\to W$ becomes a simplicial map $K\to C$
under the triangulation homeomorphisms.
Since the restriction of this map to any slice over an evenly covered subset
is invertible, the group of
deck transformations (i.e. $G$) acts simplicially on $K$.
Since $W$ is oriented, its covering space $V$ is oriented as well.
The boundary of $V$ is 
$(F|_{\partial W})^* EG = f^* EG \sqcup c^* j^* EG = X \sqcup X'.$
Thus $V$ is a simplicial Witt $G$-bordism between a multiple $rX$ of $X$
and $X'$. Now $j^* EG = \intg/_k,$ the principal $G$-bundle over a point.
Thus
$X' = c^* (\intg/_k) = Y \times \intg/_k,$
where $\intg/_k$ acts trivially on $Y$ and by translation on the second factor.
Hence by Proposition \ref{prop.wittgbordisminvariance}, 
$r\Sign (g,X) = \Sign (g, rX) = \Sign (g,X') = \Sign (g,Y\times \intg/_k)=0 \in \cplx$.
This implies that $\Sign (g,X)=0$.
\end{proof}

When $\pi$ is a regular covering projection,
Theorem \ref{thm.lclassorbitspace} together with Theorem 
\ref{thm.freeaction} asserts that
\[ L_* (X/G) =  \frac{1}{|G|} \pi_* L_*(X). \] 
This was already established in \cite{banaglcovertransfer}.

\end{document}